\documentclass[11pt]{amsart}
\usepackage[utf8]{inputenc}
\usepackage{a4wide}
\usepackage{amsmath,mathrsfs,eucal}
\usepackage{xcolor}
\usepackage{bbm,stix}
\usepackage{enumitem}
\usepackage{graphics} 
\usepackage{epsfig} 
\usepackage{graphicx,subfig}                     
\usepackage{wrapfig}
\usepackage{textcomp}
\usepackage{tikz}
\usepackage{pgfplots}
\usepackage{dsfont}

\usepackage{multicol}  
\usepackage{epstopdf}
\usepackage{hyperref}

\numberwithin{equation}{section}

\newcommand{\N}{\mathbb{N}}
\newcommand{\R}{\mathbb{R}}

\newcommand{\calC}{\mathcal{C}}
\newcommand{\calD}{\mathcal{D}}
\newcommand{\calF}{\mathcal{F}}
\newcommand{\calI}{\mathcal{I}}
\newcommand{\calK}{\mathcal{K}}
\newcommand{\calM}{\mathcal{M}}
\newcommand{\calR}{\mathcal{R}}

\newcommand{\bbmI}{\mathbbm{1}}

\newcommand{\sfX}{\mathsf{X}}
\newcommand{\sfF}{\boldsymbol{\mathsf{F}}}

\newcommand{\sfR}{\mathsf{R}}

\newcommand{\Lip}{\text{Lip}}
\newcommand{\sic}{\eta'_c(\rho_{i}^h)}
\newcommand{\sicm}{\eta'_c(\rho_{i-1}^h)}

\newcommand{\x}{\mathbf{x}}
\newcommand{\f}{\,\mathbf{f}}
\renewcommand{\j}{\mathbf{j}}
\newcommand{\intK}{\textrm{int\,}\calK}

\DeclareMathOperator*{\sign}{sign}

\newtheorem{mainthm}{Theorem}

\newtheorem{thm}{Theorem}[section]
\newtheorem{lemma}[thm]{Lemma}
\newtheorem{prop}[thm]{Proposition}

\newtheorem{remark}[thm]{Remark}

\theoremstyle{definition}
\newtheorem{defi}[thm]{Definition}

\title[On gradient flow and entropy solutions for transport equations]{On gradient flow and entropy solutions for nonlocal transport equations with nonlinear mobility}
\author{Simone Fagioli \and Oliver Tse}

\address[Simone Fagioli]{\newline Dipartimento di Ingegneria e Scienze dell’Informazione e Matematica\newline Universit\`a degli Studi dell’Aquila, Via Vetoio 1, 67100 Coppito, L’Aquila, It.}
\email{simone.fagioli@univaq.it}
\address[Oliver Tse]{\newline Department of Mathematics and Computer Science\newline
Eindhoven University of Technology, 5600 MB Eindhoven, The Netherlands}
\email{o.t.c.tse@tue.nl}
\date{\today}
\keywords{Deterministic particle approximation; Entropy solutions; Gradient flow; Nonlocal transport equations.} 
\subjclass[2010]{35A15;35A35;35D30;45K05;65M75}

\begin{document}

\maketitle
\begin{abstract}
  We prove the well-posedness of entropy solutions for a wide class of nonlocal transport equations with nonlinear mobility in one spatial dimension. The solution is obtained as the limit of approximations constructed via a deterministic system of interacting particles that exhibits a gradient flow structure. At the same time, we expose a rigorous gradient flow structure for this class of equations in terms of an Energy-Dissipation balance, which we obtain via the asymptotic convergence of functionals.
\end{abstract}


\section{Introduction}

The main object of study in this article is the well-posedness of the following PDE
\begin{equation}\label{eq:main}
    \partial_t \rho_t = \partial_x \left(\vartheta(\rho_t)\, \partial_x \calF'(\rho_t)\right)\qquad\text{in $(0,\infty)\times\R$\,,}
\end{equation}
where $\vartheta:[0,\infty)\to[0,\infty)$ is a given {\em mobility} function, and $\calF$ is a driving energy functional of the form
\[
    \calF(\rho) = \int_\R V(x)\rho(x)\,dx + \frac{1}{2}\iint_{\R\times\R} W(x-y) \rho(x)\rho(y)\,dx\,dy\,,
\]
with external potential $V:\R\to\R$ and interaction potential $W:\R\to\R$. Here, $\calF'(\rho)$ denotes the variational derivative of $\calF$, which in this case, is given by $\calF'(\rho) = V + W\ast\rho$. Equation  \eqref{eq:main} should be understood in the distributional sense and is subjected to the initial datum $\rho(0,\cdot)=\bar\rho$. In this work, we are interested in mobility functions $\vartheta$ satisfying $\vartheta(0)=0$ and $\vartheta(M)=0$ for some $M>0$.

When $\vartheta(\rho)=\rho$, the equation reduces to a nonlocal transport equation, and the special case when $W$ is the Newtonian potential is of particular interest in the field of superconductivity and superfluidity \cite{AMS2011,AS2008,BLL2012,LZ2000,SV2014}. In this case, \eqref{eq:main} is known to possess a 2-Wasserstein gradient flow structure, and the, by now standard, AGS theory for gradient flows in metric spaces may be employed to study such equations for a general class of driving functionals $\calF$ \cite{AGS2008,CdiFFLS2011}. One way of rigorously expressing \eqref{eq:main} as a 2-Wasserstein gradient flow is by means of a family of evolution variational inequalities. This relies on the $\lambda$-convexity of the driving energy functional $\calF$ along geodesics of the 2-Wasserstein distance. In addition to well-posedness results, a zoo of numerical schemes have been developed that exploit the gradient structure of the equation (see \cite{CMW2021} for an overview).

When $\vartheta$ is nonlinear, equation \eqref{eq:main} can be derived from models that often appear in the transport phenomena of biological systems with overcrowding prevention \cite{BdiFD2006,CR2006,diFR2008,DB2015}, and in the studies of phase segregation, relaxation of fermionic gas and vortex formation in mathematical physics \cite{CG-CV2019,CG-CV2020,DGM2013,GL1998,Kani1995,Ott1999,Slep2008}. In contrast to the linear case, the geometric nature of \eqref{eq:main} is less understood (see \cite{diFra2016} for a short review) when $\vartheta$ is nonlinear. For example, when $V(x)=-x$ and $W\equiv 0$, \eqref{eq:main} reduces to the scalar conservation law
\[
    \partial_t\rho_t + \partial_x \vartheta(\rho_t) = 0\,,
\]
where $\vartheta$ now plays the role of the flux function. Therefore, providing a gradient flow structure to \eqref{eq:main} is akin to providing a gradient flow structure for a scalar conservation law. 

Since the early 2000s, researchers have been working to establish links between conservation laws and Wasserstein metrics. Two strands of research stand out: (a) contraction properties of solutions in Wasserstein spaces \cite{BBL2005,CdiFL2006,EGO2016}; and (b) the use of generalized transport distances with nonlinear mobility as metrics \cite{CLSS2010,DNS2009,LM2010}. Other notable works are \cite{BCdiFP2015,Bre2009}, where equivalent characterisations of entropy solutions to conservation laws are discussed. To the best of our knowledge, a general method to rigorously make sense of a scalar conservation law as a gradient flow is still missing.

Our work, which takes a different route, not only sheds new light onto the unveiling of a gradient flow structure for \eqref{eq:main}, but also provides the well-posedness (in the sense of entropy solutions) of \eqref{eq:main} for a general class of interaction potentials, including the Newtonian potential, which we believe to be novel. This route is inspired by deterministic particle approximations (DPA) designed for equations like \eqref{eq:main} \cite{diFFRo2017,diFFRa2019,diFR2008,FRa2018}, by recent advances in the theory of generalized gradient structures \cite{EPSS2021,PRST2020}, and by the asymptotic limits of such structures \cite{Mielke2016,SS2004}. More explicitly, we introduce a DPA for the approximation of \eqref{eq:main} that possesses a generalized gradient structure, and show that this approximation converges to the unique entropy solution of \eqref{eq:main}, which also exhibits a gradient structure. This choice of approximation restricts our analysis to the 1-dimensional setting, because its generalization to higher dimensions is not straightforward.

\medskip

We informally illustrate the methodology in more detail in the following.

\subsection{Background and methodology}\label{sec:method} We begin by noticing that \eqref{eq:main} may be expressed as
\begin{align}
	\partial_t\rho_t +\partial_x j_t &= 0\,, \tag{{CE}}\label{eq:intro-continuity} \\
	j_t &= -\vartheta(\rho_t)\,\partial_x\calF'(\rho_t)\,,\tag{{FF}}\label{eq:intro-force-flux}
\end{align}
where \eqref{eq:intro-continuity} says that the {\em density-flux} pair $(\rho,j)$ satisfies the {\em continuity equation}, while equation \eqref{eq:intro-force-flux} describes the relationship between the {\em force} $\partial_x\calF'(\rho)$ and the {\em flux} $j$, which we call the {\em force-flux relation}. By introducing a {\em dual dissipation potential} $\calR^*: L^1(\R)\times L^\infty(\R)\to[0,+\infty]$,
\[
	\calR^*(\rho,\xi) = \frac{1}{2}\int_\R |\xi(x)|^2\, \vartheta(\rho(x))\,dx\,,
\]
one notices that the force-flux relation \eqref{eq:intro-force-flux} may be expressed as
\[
	j = D_2\calR^*(\rho,-\partial_x \calF'(\rho))\,.
\]
Via Legendre-Fenchel duality, we obtain a variational characterization of \eqref{eq:intro-force-flux}:
\begin{align}\label{eq:variational-force-flux}
    \calR(\rho,j) + \calR^*(\rho,-\partial_x\calF'(\rho)) = \langle j,-\partial_x\calF'(\rho)\rangle\,,
\end{align}
where the {\em dissipation potential} $\calR$ is the Legendre dual of $\calR^*$ w.r.t.\ its second argument, i.e.,\
\[
	\calR(\rho,j) = \frac{1}{2}\int_\R \left| \frac{j(x)}{\vartheta(\rho(x))}\right|^2 \vartheta(\rho(x))\,dx\,.
\]

This dissipation potential $\calR$ gives rise to a generalized transport cost that falls within the theory developed in \cite{DNS2009,LM2010}, and was used in \cite{CLSS2010} to construct gradient flow solutions for equations of the type \eqref{eq:main} whose driving functionals contain only the internal energy. In fact, it was formally shown in \cite{CLSS2010} that the driving functional $\calF$ we consider here cannot be $\lambda$-convex unless $\vartheta$ is linear, thus invalidating their method for nonlinear mobilities---this suggests the need for a different method.

Our approach does not rely on the $\lambda$-convexity of $\calF$ and makes sole use of the variational form of the force-flux relation \eqref{eq:variational-force-flux}. Observe that for any sufficiently smooth density-flux pair $(\rho,j)$ satisfying \eqref{eq:intro-continuity}, the right-hand side of \eqref{eq:variational-force-flux} may be expressed, via the chain rule, as
\begin{align}\label{eq:intro-chainrule}\tag{{CR}}
	\langle j_t,-\partial_x\calF'(\rho_t)\rangle = -\langle \partial_x j_t,-\calF'(\rho_t)\rangle = -\langle \partial_t\rho_t,\calF'(\rho_t)\rangle = -\frac{d}{dt}\calF(\rho_t)\,.
\end{align}
Integrating \eqref{eq:variational-force-flux} over any interval $(s,t)\subset(0,T)$ yields the {\em Energy-Dissipation balance}
\begin{align}\label{eq:intro-EDB}\tag{{EDB}}
	\int_s^t \calR(\rho_r,j_r) + \calR^*(\rho_r,-\partial_x\calF'(\rho_r))\,dr = \calF(\rho_s) - \calF(\rho_t)\,.
\end{align}
Equations \eqref{eq:intro-continuity} and \eqref{eq:intro-EDB} then allows us to define a notion of gradient flow solution for \eqref{eq:main}. Morally, any pair $(\rho,j)$ satisfying \eqref{eq:intro-continuity} and \eqref{eq:intro-EDB} is said to be an $(\calF,\calR,\calR^*)$-gradient flow solution of \eqref{eq:main}, if it satisfies, additionally, the chain rule \eqref{eq:intro-chainrule}. 
Notice that \eqref{eq:intro-chainrule} and \eqref{eq:intro-EDB} characterize the flux $j$ in terms of the force $\partial_x\calF'(\rho)$ via the dual dissipation potential $\calR^*$, i.e.\ together, they give \eqref{eq:intro-force-flux}, and hence a distributional solution for \eqref{eq:main}.

\medskip

As mentioned above, our method for constructing an $(\calF,\calR,\calR^*)$-gradient flow solution $(\rho,j)$ of \eqref{eq:main} relies on a deterministic particle approximation for \eqref{eq:main}. More concretely, we consider a microscopic system of $N\in\N$ {\em particles} $\x = (\x_0,\ldots,\x_N)^\top\in\R^N$ that evolves according to the system
\begin{equation}\tag{{DPA}}\label{eq:particle_complete}
\left\{\qquad\begin{aligned}
    \dot{\x}_0 &= -\beta(\rho_0^h)\f_0^- - \beta(0)\f_0^+\\
    \dot\x_i &= -\beta(\rho_i^h)\f_i^--\beta(\rho_{i-1}^h)\f_i^+ \\
    \dot{\x}_N &= -\beta(0)\f_N^- -\beta(\rho_{N-1}^h)\f_N^+
\end{aligned}\qquad\text{for $i=1,\ldots,N-1$}\,,\right.
\end{equation}
subjected to the initial condition $\bar \x$ defined in \eqref{eq:dscr_IC_un}. Here, $\beta(\rho):=\vartheta(\rho)/\rho$,
\begin{equation*}
    \rho_i^h (t) = \frac{h}{\x_{i+1}(t)-\x_i(t)},\qquad i=0,\ldots,N-1,\qquad h=\frac{m}{N}\,,
\end{equation*}
with $m>0$ being the initial mass of the density,
\begin{equation*}
	\f_i = V'(\x_i) + \sum\nolimits_{j\ne i} h W'(\x_i-\x_j)\,,\qquad i=0,\ldots,N\,,
\end{equation*}
and where, $\f_i^+$ and $\f_i^-$ denote the positive and negative part of $\f_i$, i.e. $\f_i=\f_i^++\f_i^-$.

The well-posedness of \eqref{eq:particle_complete} for a wide class of interaction potentials, including the Newtonian potential, and a priori estimates for the solution is given in Section~\ref{sec:DPA}. It is also shown there that the system \eqref{eq:particle_complete} exhibits a generalized gradient structure, i.e.\ there is a triplet $(\calF_h,\calR_h, \calR_h^*)$ such that the unique solution $\x$ of \eqref{eq:particle_complete} is a $(\calF_h,\calR_h, \calR_h^*)$-gradient flow solution in the sense given above with the continuity equation $\dot\x = \j$ and force-flux relation $\j=D_2\calR_h^*(\x,-\calF_h'(\x))$.

\medskip

Having solved this system for $\x$, we then consider the pair $(\hat\rho^h,\hat\jmath^h)$ given by
\begin{equation}\label{eq:piece_den_intro}
	\hat\rho_t^h(x) := \sum_{i=0}^{N-1} \rho_i^h(t) \bbmI_{K_i(t)}(x)\,,\qquad \hat \jmath_t^h(x) := \sum_{i=0}^{N-1} \rho_i^h(t)\, u_i^h(t,x) \bbmI_{K_i(t)}(x)\,,
\end{equation}
where $K_i$ is the interval $K_i = (\x_i,\x_{i+1})$ with length $|K_i(t)|:=|\x_{i+1}(t) - \x_i(t)|$, and
\[
	u_i^h(t,x) := \frac{\x_{i+1}(t)-x}{|K_i(t)|}\dot \x_i(t) + \frac{x-\x_i(t)}{|K_i(t)|}\dot \x_{i+1}(t)\qquad \text{for $x\in K_i$,\; $i=0,\ldots,N-1$\,.}
\]
By construction, the pair $(\hat\rho^h,\hat\jmath^h)$ satisfies the continuity equation \eqref{eq:intro-continuity} (cf.\ Lemma~\ref{lem:reconstruction-continuity}). Under appropriate assumptions on the initial datum $\bar\rho$, we prove in Section~\ref{sec:Continuous} that the sequence $(\hat\rho^h,\hat\jmath^h)_{h>0}$ admits an accumulation point $\rho\in L^1((0,T)\times\R)$ with respect to the strong topology, and a corresponding flux $j$. In Section~\ref{sec:convergence}, this accumulation point is shown to be an $(\calF,\calR,\calR^*)$-gradient flow solution in the sense described above, as well as an entropy solution of \eqref{eq:main} in the Kružkov sense \cite{Kru1970}.
Since entropy solutions are unique, we ultimately obtain the convergence of the full sequence $(\hat\rho^h)_{h>0}$ to the entropy solution of \eqref{eq:main}, and consequently also the gradient flow solution.

\medskip

We stress once more that this work provides not only the well-posedness and the uncovering of a gradient flow structure for \eqref{eq:main}, but it also introduces a simple numerical scheme for the approximation of these solutions. While the results in this paper are stated for the full space $\R$, the strategy can be easily adapted to bounded intervals with different boundary conditions (eg.\ Dirichlet, Neumann, periodic boundary). A clear drawback, however, of the current deterministic particle approximation lies in the fact that it works only in one spatial dimension.

\subsection{Notation, assumptions and definitions}\label{sec:defi} 
This section is devoted to setting the notation, assumptions and introducing definitions that will be used throughout the paper.

\subsubsection*{Notation} With the symbol $C_c(\R)$ we will denote the space of continuous functions on $\R$ with compact support, and with $C_0(\R)$ its completion with respect to the supremum norm. We denote by $\calM(\R)$ the space of Borel measures $\mu$ on $\R$ with finite total variation norm
\[
    |\mu|(\R) :=\sup\left\{ \int_\R \varphi\,d\mu\,:\, \varphi\in C_0(\R),\; \|\varphi\|_{\infty}\le 1\right\},
\]
and by $\calM^+(\R)$ the collection of $\mu\in\calM^+(\R)$ that are non-negative.  It follows from the Riesz theorem that $\calM(\R)$ can be identified with the dual of $C_0(\R)$ by the duality pairing
\[
    \langle \varphi,\mu\rangle :=\int_\R \varphi\,d\mu\qquad\text{for all $\varphi\in C_0(\R)$}\,.
\]
We will consider the corresponding weak-$*$ topology on $\calM(\R)$, where precompactness in this topology is equivalent to boundedness in the total variation norm.

We also consider the space of functions with bounded variation
\[
    BV(\R) := \Bigl\{ f\in L^1(\R)\,:\, Df\in\calM(\R)\Bigr\}\,,
\]
equipped with the norm
\[
    \|f\|_{BV(\R)} := \|f\|_{L^1(\R)} + |Df|(\R)\,,
\]
where $Df$ denotes the distributional derivative $f$.

\medskip

\subsubsection*{Assumptions} We first set the assumptions on the initial datum $\bar{\rho}$. In all cases, we will assume that
\begin{itemize}
\item[(In1)] $\bar{\rho} \in BV(\R)\cap L^\infty(\R)$, $\bar\rho\ge 0$, with finite mass $\int_\R \bar{\rho}(x)\, dx = m$, for some $m>0$, and has compact support.
\end{itemize}
In certain situations, we will need the following additional assumption on $\bar\rho$:
\begin{itemize}
    \item[(In2)] There is some $\sigma >0$ such that $\bar{\rho}(x)\ge \sigma$ for every $x\in \text{supp}(\bar\rho)$.
\end{itemize}

A crucial role is played by the map $\beta:[0,\infty)\to[0,\infty)$ that we assume to satisfy:
\begin{itemize}
\item[(A-$\beta$)] $\beta \in \Lip([0,+\infty))$ is a decreasing function such that $\beta(0)=\beta_{max}>0$, $\beta(M_\beta)=0$ and $\beta\equiv 0$ on $[M_\beta,+\infty)$ for some $M_\beta>0$. 
\end{itemize}

\begin{remark}
    The assumption (A-$\beta$) on $\beta$ includes maps of the form
    \[
      \beta(s)=(M_\beta-s^\gamma)_{+}, \qquad \gamma\geq 1,
    \]
    that commonly appear in traffic flow or in biological models \cite{BdiFD2006,CR2006,DB2015}. Such types of maps $\beta$ give rise to mobility functions $\vartheta$ that are concave. The concavity of $\vartheta$ is essential in the works \cite{CLSS2010,DNS2009,LM2010}, but not required in our approach.
\end{remark}

As mentioned in the introduction, the driving energy functional takes the form
\begin{equation}\label{eq:energy_fun_intro}
    \calF(\rho) = \int_\R V(x)\rho(x)\,dx + \frac{1}{2}\iint_{\R\times\R} W(x-y) \rho(x)\rho(y)\,dx\,dy\,,
\end{equation}
where the external potential $V:\R\to\R$ is assumed to satisfy the following property:
\begin{itemize}
    \item[(A-V)] $V \in C^2(\R)$ is an external potential with $V'$ having linear growth, i.e.
    \begin{equation}\label{eq:linear-growth}
        |V'(r)| \le c_V\bigl(1 + |r|\bigr),\qquad |r|>0\,,
    \end{equation}
    for some constant $c_V>0$.
\end{itemize}
As for the interaction potential $W$, we will distinguish between three different classes. We assume that $W:\R\to\R$ is a radially symmetric nonlocal potential and we further say that
\begin{itemize}
    \item[(A-Wr)] $W$ is a {\em regular potential} if $W\in\calC^1(\R)$ with $W'$ in the Sobolev space $W^{2,\infty}(\R)$ and having linear growth \eqref{eq:linear-growth} with some constant $c_W>0$.

   \item[(A-Wn)] $W$ is a {\em Newtonian potential}, namely $W(x)=\pm |x|$, where the sign depends on whether it is attractive or repulsive;
   \item[(A-Wm)] $W$ is a {\em mild potential} if $W\in \Lip_{loc}(\R)$ with $W'$ in the Sobolev space $W^{2,\infty}(\R\setminus\{0\})$ and having linear growth \eqref{eq:linear-growth} with some constant $c_W>0$.
\end{itemize}

\begin{remark}
    Notice that interaction potentials $W$ satisfying (A-Wr) or (A-Wn) are contained in (A-Wm), i.e.\ regular potentials and Newtonian potentials are mild potentials.
\end{remark}

In the following we will make use of the following compact notation for the force:
\[
     \sfF_\rho(t,x) := \partial_x\calF'(\rho) =  V'(x) + (W'\star\rho_t)(x)\,.
\]

\begin{remark}\label{rem:Newtonian-Lipschitz}
    If $W$ satisfies (A-Wn), then $\sfF_\rho$ takes the simple form
    \begin{align}\label{eq:newtonian-force}
        \sfF_\rho(t,x) = V'(x) \,\pm\, \left(2\int_{-\infty}^x\rho_t(y)\,dy - m\right),
    \end{align}
    where the sign depends on whether $W$ is attractive or repulsive. In particular, $x\mapsto\sfF_\rho(t,x)$ is globally Lipschitz for almost every $t\in [0,T]$ whenever $\rho_t\in L^\infty(\R)$.
\end{remark}

\begin{remark}
    Interaction potentials $W$ satisfying (A-Wm) include Morse type potentials of the form
    \[
        W(x) = -c_A e^{-|x|/\ell_A} + c_R e^{|x|/\ell_R}\,,
    \]
    where $c_A>0$, $c_R>0$ are the attractive and repulsive strengths respectively, and $\ell_A>0$, $\ell_R>0$ are their respective length scales \cite{CRP2013}.
\end{remark}

\subsubsection*{Definitions} In the following, we define what we mean by a gradient flow solution of \eqref{eq:main}, and we recall the notion of entropy solutions (in the Kružkov sense) to such equations. We begin with the definition of a pairs $(\rho,j)$ satisfying the continuity equation introduced above. 


\begin{defi}\label{def:cont-eq}
    We denote by $\mathcal{CE}(0,T)$ the set of pairs $(\rho,j)$ given by
    \begin{enumerate}[label=(\roman*)]
        \item a weakly-$*$ continuous curve $[0,T]\ni t\mapsto \rho_t\in\calM^+(\R)$, and
        \item a measurable family $j=\{j_t\}_{t\in[0,T]}\subset \calM(\R)$ with $\int_0^T |j_t|(\R)\,dt<\infty$\,,
    \end{enumerate}
    satisfying the continuity equation
	\begin{align}\label{eq:cont-eq}
    	\partial_t \rho_t + \partial_x j_t &= 0\,,\tag{{CE}} 
	\end{align}
	in the following sense: For any $0\le s<t\le T$,
	\[
		\langle \varphi,\rho_t\rangle - \langle \varphi,\rho_s\rangle = \int_s^t \langle \partial_x\varphi,j_r\rangle\,dr\qquad\text{for all $\varphi\in\Lip_b(\R)$}\,.
	\]
	Here, $\Lip_b(\R)$ denotes the space of bounded Lipschitz functions.
\end{defi}

\begin{defi}[Gradient flow solutions]\label{def:grad_flow}
    A pair $(\rho,j)\in\mathcal{CE}(0,T)$ is said to be an $(\calF,\calR,\calR^*)$-gradient flow solution of \eqref{eq:main} with initial datum $\bar\rho\in \calM^+(\R)$ satisfying (In1), if is satisfies 
    \begin{enumerate}[label=(\roman*)]
        \item $\sup_{t\in[0,T]} \text{supp}(\rho_t)<\infty$;
        \item $\rho_t \rightharpoonup^* \bar\rho$ weakly-$*$ in $\calM^+(\R)$ (i.e.\ in duality against $C_0(\R)$ functions) as $t\to 0$;
        \item the Energy-Dissipation balance
        \begin{align}\tag{{EDB}}
	\int_s^t \calR(\rho_r,j_r) + \calR^*(\rho_r,-\sfF_\rho(r,\cdot))\,dr = \calF(\rho_s) - \calF(\rho_t)\,.
\end{align}
        for any interval $(s,t)\subset[0,T]$; and
        \item the chain rule
        \begin{align}\tag{{CR}}
	\frac{d}{dt}\calF(\rho_t) = \langle \sfF_{\rho}(t,\cdot),j_t\rangle\qquad\text{for almost every $t\in(0,T)$}\,.
\end{align}
    \end{enumerate}
\end{defi}

\begin{remark}\label{rem:weak-solution}
We recall from the discussion above that $(\calF,\calR,\calR^*)$-gradient flow solutions of \eqref{eq:main} are also weak solutions in the sense that $\rho$ satisfies
\[
    \langle \varphi,\rho_t\rangle - \langle \varphi,\rho_s\rangle = -\int_s^t \langle \partial_x\varphi,\vartheta(\rho_r)\,\partial_x\calF'(\rho_r)\rangle\,dr\qquad\text{for any $\varphi\in\Lip_b(\R)$,  $(s,t)\subset(0,T)$}\,.
\]
Indeed, the Energy-Dissipation balance and chain rule implies that
\[
    j_t(dx) = -\vartheta(\rho_t(x))\,\sfF_\rho(t,x)dx\qquad\text{for almost every $t\in[0,T]$}.
\]
The fact that the pair $(\rho,j)$ satisfies the continuity equation \eqref{eq:cont-eq} then yields the assertion.
\end{remark}

\begin{defi}[Entropy solutions]\label{def:entropy_sol}
Let $\bar{\rho}\in L^1\cap L^\infty(\R)$. We call a measurable family $\{\rho_t\}_{t\in[0,T]}\subset L^1(\R)$ an entropy solution of \eqref{eq:main} with initial datum $\bar{\rho}$ if $\rho\in L^\infty([0,T]; L^1(\R))$ and, for any constant $c>0$ and all $\varphi \in C_c^\infty([0,T)\times\R)$ with $\varphi\geq 0$, the following {\em entropy inequality} holds:
\begin{align}\label{eq:EI}\tag{EI}
\begin{aligned}
   0\leq &\int_\R |\bar\rho- c|\varphi(0,\cdot) \,dx\\
   & +  \int_0^T\int_\R |\rho_t- c|\partial_t\varphi-\sign(\rho_t-c)\left[\left(\vartheta(\rho_t)-\vartheta(c)\right)\sfF_\rho\,\partial_x\varphi-\vartheta(c)\,\partial_x\sfF_\rho\,\varphi\right] dx \,dt.
\end{aligned}
\end{align}
\end{defi}

\medskip

This definition of entropy solutions coincides with the one used in \cite{diFFRa2019}, and is contained within the well-known notion of entropy solutions by Kružkov \cite{Kru1970} (see also \cite{KR2003,Vol1967}) that is known to assert the uniqueness of solutions to \eqref{eq:main}.

\subsection{Origins of the DPA and connections to related schemes}
A formal way for discovering the relation between \eqref{eq:main} and \eqref{eq:particle_complete} is to formulate the equation for $\rho$ in terms of its {\em pseudo-inverse distribution function}.  We  briefly describe the idea below: let $\sfR$ be the cumulative distribution function associated to $\rho$ and define $\sfX$ to be the corresponding pseudo-inverse function, defined by
 \[
  \sfX(z)=\inf_{x\in\R}\left\{\sfR(x)>z\right\},\qquad \mbox{for }z\in\left[0,m\right],
 \]
 where $m$ denotes the total mass of $\rho$, i.e.\ $m=\int_\R \rho\,dx$.
 Formally, $\sfX$ satisfies the following PDE
 \begin{equation}\label{eq:pseudo}
     \partial_t \sfX  = -\beta\left(\frac{1}{\partial_z\sfX(t,z)}\right) \sfF_\rho(t,\sfX(t,z))\,,
 \end{equation}
 with
 \[
 \sfF_\rho(t,\sfX(t,z))=V'(\sfX(t,z))+\int_0^m W'(\sfX(t,z)-\sfX(t,\zeta))\,
  d\zeta.
 \]
  For $N\in\mathbb{N}$, we consider a uniform partition of the interval $\left[0,m\right]$, $z_i=hi$, $i=0,\ldots,N$ with $h=m/N$ being the size of each partition. By calling $\sfX(t,z_i)=\sfX_i(t)$, a forward finite difference discretization on \eqref{eq:pseudo} yields
 \begin{equation}\label{eq:pseudo_dif}
 \dot\sfX_i(t) = -\beta\left(\frac{h}{\sfX_{i+1}(t)-\sfX_{i}(t)}\right)\sfF_\rho(t,\sfX_i(t)).
 \end{equation}
In order to get a better numerical approximation of \eqref{eq:pseudo}, for example in case of a force field $\sfF_\rho(t,x)$ with possible changes in sign, we can consider the following \emph{upwind}-type discretization
 \begin{equation}\label{eq:pseudo_vol}
 \dot\sfX_i(t) = -\beta\left(\frac{h}{\sfX_{i+1}(t)-\sfX_{i}(t)}\right)\sfF_\rho(t,\sfX_i(t))^--\beta\left(\frac{h}{\sfX_{i}(t)-\sfX_{i-1}(t)}\right)\sfF_\rho(t,\sfX_i(t))^+,
 \end{equation}
which results in \eqref{eq:particle_complete} after a further approximation of the integral term in $\sfF_\rho$.

Let us stress that the idea of using the pseudo-inverse formulation in order to obtain a numerical scheme for the corresponding PDEs dates back to the works \cite{GT1,GT2}, where the authors present a time-space discretization of the pseudo-inverse equation related to the classical porous medium equation (see also \cite{CDMM,CHPW}). Convenient modifications of \eqref{eq:pseudo_dif} were used in \cite{diFFRa2019,diFSt} in order to prove convergence of the resulting \eqref{eq:particle_complete} to equations in form \eqref{eq:main} with only attractive non-local kernels $W$, \cite{diFFRa2019}, or only external forces $V$, \cite{diFSt}. In particular in \cite{diFSt}, the authors needs to design different schemes for different choice on $V$ depending on its sign. 

The DPA can also be related to the moving mesh schemes  \cite{BHJ11,BHR96,BHR09,CHR02,CHR03,SMR01} that are largely applied in several context such as diffusion problems, computational fluid dynamics and scalar conservation laws. Let us stress once more that the one-dimesional limitation is present also for those scheme, in particular in its application to scalar conservation laws.

Even if a complete numerical analysis of \eqref{eq:particle_complete}, i.e.\ stability and error analysis, is out of the scope of this paper, we would like to emphasize some of the advantages of the method, that are: (a) its flexibility in including a large class of potentials and kernels $V$ and $W$, (b) the natural energy-dissipation property of \eqref{eq:particle_complete}, and (c) its easy implementation, see Sections \ref{sec:discrete-GF} and \ref{sec:num} respectively. While properties (a) and (c) are present in many other numerical scheme, we believe that property (b) is one that is uncommon and sets our scheme apart from the rest.

\subsection{Main results}
We are now in the position to state our main results.

\begin{mainthm}\label{mainthm:regular-newtonian}
    Let $\bar\rho$, $\beta$ and $V$ satisfy (In1), (A-$\beta$) and (A-V) respectively. If $W$ satisfies either (A-Wr) or (A-Wn), then for any $T>0$, the sequence of piecewise constant interpolation $\{\hat{\rho}^h\}_{h\in(0,1)}$ defined in \eqref{eq:piece_den_intro} converges in $L^1\left([0,T]\times \R\right)$ to a function $\rho$ as $h\to 0$. Moreover, 
    \begin{enumerate}[label=(\roman*)]
        \item there exists a measurable family $j=\{j_t\}_{t\in[0,T]}$ such that the pair $(\rho,j)\in\mathcal{CE}(0,T)$ and is an $(\calF,\calR,\calR^*)$-gradient flow solution of \eqref{eq:main} with initial datum $\bar\rho$ in the sense of Definition~\ref{def:grad_flow};
        \item $\rho$ is the unique entropy solution of $\eqref{eq:main}$ in the sense of Definition~\ref{def:entropy_sol}.
    \end{enumerate} 
    Furthermore, the solution satisfies the bound
    \[
        \rho(t,x) \le \max\bigl\{\|\bar\rho\|_{L^\infty(\R)},M_\beta\bigr\}\qquad\text{for almost every $(t,x)\in[0,T]\times\R$}\,.
    \]
\end{mainthm}

Theorem~\ref{mainthm:regular-newtonian} adds to the current understanding of \eqref{eq:main} in various ways. Firstly, Theorem~\ref{mainthm:regular-newtonian}(i) provides a rigorous variational description of solutions to \eqref{eq:main} in terms of $(\calF,\calR,\calR^*)$-gradient flows. Putting aside the novelty of Theorem~\ref{mainthm:regular-newtonian}(i), the existence of entropy solutions to \eqref{eq:main} for Newtonian interaction potentials $W$ was not known before Theorem~\ref{mainthm:regular-newtonian}(ii), let alone the convergence of a numerical approximation towards entropy solutions thereof.

\medskip

When $W$ is only assumed to satisfy (A-Wm), the corresponding force $\sfF_\rho$ may not be regular enough, thus prohibiting us from obtaining the entropy inequality \eqref{eq:EI}---the inequality requires $\partial_x\sfF_\rho$ to be well-defined. Nevertheless, we still obtain a gradient flow solution of \eqref{eq:main}, albeit under a uniform lower bound assumption on the initial datum $\bar\rho$.

\begin{mainthm}\label{mainthm:mild}
    Let $\bar\rho$, $\beta$, $V$ and $W$ satisfy (In1), (A-$\beta$), (A-V) and (A-Wm) respectively. If $\bar\rho$ satisfies additionally (In2), then for any $T>0$, the sequence of piecewise constant interpolation $\{\hat{\rho}^h\}_{h\in(0,1)}$ defined in \eqref{eq:piece_den_intro} converges in $L^1\left([0,T]\times \R\right)$ to a function $\rho$ as $h\to 0$, satisfying the bound
    \[
        \rho(t,x) \le \max\bigl\{\|\bar\rho\|_{L^\infty(\R)},M_\beta\bigr\}\qquad\text{for almost every $(t,x)\in[0,T]\times\R$}\,.
    \]
    Furthermore, there exists a measurable family $j=\{j_t\}_{t\in[0,T]}$ such that the pair $(\rho,j)\in\mathcal{CE}(0,T)$ is an $(\calF,\calR,\calR^*)$-gradient flow solution of \eqref{eq:main} with initial datum $\bar\rho$.
\end{mainthm}

While Theorem B provides a positive result, it is unsatisfactory on two (possibly more) accounts: we are uncertain if $(\calF,\calR,\calR^*)$-gradient flow solutions are unique, and if the uniform lower bound assumption on the initial datum is indispensable. On the other hand, it was shown in \cite{diFFRa2019} that when the support of the initial datum $\bar\rho$ is disconnected---hence violating (In2)---one can construct multiple weak solutions to \eqref{eq:main}. For this reason, we conjecture that the current definition of a gradient flow solution to \eqref{eq:main} allows for non-uniqueness, but we digress for now, and choose to postpone further discussion on the matter to future investigation.

\medskip

The proof of Theorems~\ref{mainthm:regular-newtonian} and \ref{mainthm:mild} are summarized in Section~\ref{sec:proof-A} and Section~\ref{sec:proof-B} respectively.

\subsection*{Acknowledgments}
S.F.\ thanks the Centre for Analysis, Scientific Computing and Applications at Eindhoven University of Technology for the hospitality received during the early stages of this work. O.T.\ acknowledges support from NWO Vidi grant 016.Vidi.189.102, ``Dynamical-Variational Transport Costs and Application to Variational Evolutions".

\section{Deterministic particle approximation (DPA)}\label{sec:DPA}
Let $\bar\rho\in L^1(\R)$ satisfy (In1), and set $\bar{x}_{min}:=\min\{x\in \text{supp}(\bar{\rho})\}$ and $ \bar{x}_{max}:=\max\{x\in \text{supp}(\bar{\rho})\}$. For a fixed integer $N\in\N$, we split $\text{supp}(\bar{\rho})$ into $N$ sub-intervals such that
the mass of $\bar{\rho}$ restricted over each interval equals $h=m/N$. In order to do this, we fix $\bar{\x}_0= \bar{x}_{min}$ and $\bar{\x}_N = \bar{x}_{max}$, and define recursively the points 
\begin{equation}\label{eq:dscr_IC_un}
\bar{\x}_i = \sup \left\lbrace x \in \R : \int_{\bar{\x}_{i-1}}^x \bar{\rho}(x)\, dx < h  \right\rbrace\qquad\text{for $i \in \{ 1,\, \ldots,\, N-1\}$\,.}
\end{equation}
It is clear from the construction that $\int_{\bar{\x}_{i-1}}^{\bar{\x}_i} \bar{\rho}(x)\, dx = h$ and $\bar{\x}_0 < \bar{\x}_1 < \ldots\, < \bar{\x}_{N-1} < \bar{\x}_N$.
\begin{figure}[htbp]
\begin{center}
\begin{tikzpicture}
\draw[->] (-2.5,0) -- (5.5,0);
\draw[->] (1.5,-0.5) -- (1.5,2) ;

\draw[scale=1,domain=-2:5,smooth,variable=\x,black] plot ({\x},{1.4*exp(-0.4*(\x-1.5)*(\x-1.5)});

\node[black,scale=1] at (2,1.6){$\bar{\rho}$};

\node[scale=1] at (-2,-0.3) {\scriptsize $\bar{\x}_0$};
\node[scale=1] at (-0.15,-0.3) {\scriptsize $\bar{\x}_1$};
\node[scale=1] at (0.45,-0.3) {\scriptsize $\bar{\x}_2$};
\node[scale=1] at (0.85,-0.3) {\scriptsize $\bar{\x}_3$};
\node[scale=1] at (1.2,-0.3) {\scriptsize $\bar{\x}_4$};
\node[scale=1] at (1.5,-0.3) {\scriptsize $\bar{\x}_5$};
\node[scale=1] at (1.8,-0.3) {\scriptsize $\bar{\x}_6$};
\node[scale=1] at (2.15,-0.3) {\scriptsize $\bar{\x}_7$};
\node[scale=1] at (2.55,-0.3) {\scriptsize $\bar{\x}_8$};
\node[scale=1] at (3.15,-0.3) {\scriptsize $\bar{\x}_9$};
\node[scale=1] at (5,-0.3) {\scriptsize $\bar{\x}_{10}$};

\draw[dashed] (-0.15,0) -- (-0.15,0.5);
\draw[dashed] (0.45,0) -- (0.45,0.9);
\draw[dashed] (0.85,0) -- (0.85,1.2);
\draw[dashed] (1.2,0) -- (1.2,1.35);
\draw[dashed] (1.5,0) -- (1.5,1.4);
\draw[dashed] (1.8,0) -- (1.8,1.35);
\draw[dashed] (2.15,0) -- (2.15,1.2);
\draw[dashed] (2.55,0) -- (2.55,0.9);
\draw[dashed] (3.15,0) -- (3.15,0.5);

\draw[->] (-1,1) -- (-0.45,0.17);
\draw[->] (-1,1) -- (0.25,0.2);
\node[black,scale=1] at (-1.3,1.3){$h=\frac{m}{N}$};
\end{tikzpicture}
\caption{An example of \eqref{eq:dscr_IC_un} with $N=11$ particles. Each area between consecutive particles measures $h=m/N$.}
\label{fig:initial}
\end{center}
\end{figure}
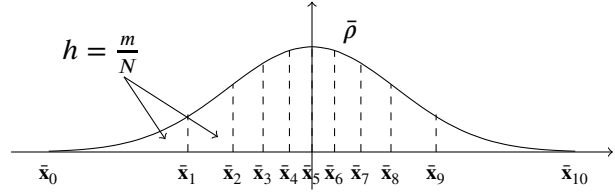

Taking the vector $\bar{\x} = (\bar{\x}_0,\ldots,\bar{\x}_N)\in\R^N$ as initial condition, we let the particles evolve according to the system of equations \eqref{eq:particle_complete}, which we recall here:
\begin{equation}\tag{{DPA}}
\left\{\qquad\begin{aligned}
    \dot{\x}_0 &= -\beta(\rho_0^h)\f_0^- - \beta_{max}\f_0^+\\
    \dot\x_i &= -\beta(\rho_i^h)\f_i^--\beta(\rho_{i-1}^h)\f_i^+ \\
    \dot{\x}_N &= -\beta_{max}\f_N^- -\beta(\rho_{N-1}^h)\f_N^+
\end{aligned}\qquad\text{for $i=1,\ldots,N-1$}\,,\right.
\end{equation}
where we have defined
\begin{equation}\label{eq:disc_density}
    \rho_i^h (t) = \frac{h}{\x_{i+1}(t)-\x_i(t)},\qquad i=0,\ldots,N-1,
\end{equation}
and
\begin{equation}\label{eq:point_functional}
	\f_i = V'(\x_i) + \sum\nolimits_{j\ne i} h W'(\x_i-\x_j)\,,\qquad i=0,\ldots,N\,,
\end{equation}
where $\f_i^+$ and $\f_i^-$ denote the positive and negative part of $\f_i$, i.e. $\f_i=\f_i^++\f_i^-$. 

In the following, we set
\[
    M:= \max\bigl\{\|\bar\rho\|_{L^\infty},M_\beta\bigr\}.
\]
The main results concerning \eqref{eq:particle_complete} are collected in the following two theorems.

\begin{thm}\label{thm:existence-DPA}
    Let $T>0$ be fixed, and let $\bar\rho$, $\beta$, $V$ and $W$ satisfy assumptions (In1), (A-$\beta$), (A-V) and (A-Wm) respectively. Then for any $h\in(0,1)$, there exists a solution $\x\in\calC([0,T];\R^{N+1})$ to the deterministic particle approximation \eqref{eq:particle_complete} having the following properties:
    \begin{enumerate}[label=(\roman*)]
        \item For every $i=0,\ldots,N-1$ and $t\in[0,T]$,
        \begin{equation}\label{eq:lower_DPA}
        |K_i|(t) = \x_{i+1}(t) - \x_i(t) \ge \frac{h}{M}\,;
    \end{equation}
        \item $\sup\nolimits_{t\in[0,T]}|\x_N(t)-\x_0(t)| <\infty$.
    \end{enumerate}
\end{thm}
\begin{proof}
    The proof of the theorem follows by combining the results obtained in Lemma~\ref{lem:lower_DPA}, Lemma~\ref{lem:local-DPA} and Lemma~\ref{lem:extend-DPA} below.
\end{proof}

\begin{thm}\label{thm:upper-bound}
    Let the assumptions of Theorem~\ref{thm:existence-DPA} hold. If $\bar\rho$ satisfies additionally (In2), then the solution provided by Theorem~\ref{thm:existence-DPA} also satisfies
    \begin{equation}\label{eq:max_part}
        |K_i|(t) = \x_{i+1}(t) - \x_i(t) \leq e^{\mu T}\frac{h}{\sigma},\qquad \text{$i=0,\ldots,N-1$, $t \in [0,T]$}\,,
\end{equation} 
for some constant $\mu>0$, independent of $h\in(0,1)$.
\end{thm}
\begin{proof}
    This is a direct result of Lemma~\ref{lem:upper-bound} below.
\end{proof}

\subsection{A priori estimates for DPA}

We now provide essential results in obtaining uniform bounds on the density, and consequently, the existence of solutions to \eqref{eq:particle_complete}, as well as compactness results that we provide in Section~\ref{sec:Continuous}. The first result provides the lower bound for the increments of $\x$ (cf.\ \eqref{eq:lower_DPA}) and is attributed solely to the monotonicity of $\beta$, i.e.\ the result is independent of forces.

\begin{lemma}\label{lem:lower_DPA}
    Let $\bar{\rho}$ and $\beta$ satisfy assumptions (In1) and (A-$\beta$) respectively. Then any solution $\x\in\calC([0,T];\R^{N+1})$ of \eqref{eq:particle_complete} satisfies
\[
 \x_{i+1}(t) - \x_i(t) \ge \frac{h}{M}\,,\qquad\forall\,\text{$i=0,\ldots,N-1$, $t \in [0,T]$}\,.
\]
\end{lemma}
\begin{proof}
We argue by contradiction. Let $\tau_1>0$ be defined by
\[
 \tau_1=\inf\left\{s\in(0,T] : \exists \,i \;\text{such that}\; \x_{i+1}(t) - \x_i(t)=\frac{h}{M}\right\}.
\]
Assume now that there exists $\tau_2\in\left(\tau_1,T\right]$ such that
\[\x_{i+1}(t) - \x_i(t)<\frac{h}{M}\qquad \forall t\in\left(\tau_1,\tau_2\right].\]
Note that $\dot{\x}_i(\tau_1)\le  0$, since
\[
 \rho_i^h(\tau_1)=M\quad \mbox{and}\quad \rho_{i-1}^h(\tau_1)\le M\,,
\]
that imply 
\[
 \beta(\rho_i^h(\tau_1))=0 \quad\mbox{and}\quad \beta(\rho_{i-1}^h(\tau_1))\ge 0.
\]
Similarly, we can argue that $\dot{\x}_{i+1}(\tau_1)\ge 0$, thus the particles $\x_i$ and $\x_{i+1}$ either remain the same distance or move away each other, at least for a small time $t>\tau_1$. 
A Taylor expansion around $\tau_1$ gives
\begin{gather*}
    \x_{i+1}(t)=\x_{i+1}(\tau_1)+\dot{\x}_{i+1}(\tilde{t}_{i+1})(t-\tau_1)\ge \x_{i+1}(\tau_1)\quad\mbox{ for all } t\in(\tau_1,\tau_1+\delta_{i+1})\,\mbox{ and } \tilde{t}_{i+1}\in(\tau_1,t)\,,\\
    \x_{i}(t)=\x_{i}(\tau_1)+\dot{\x}_{i}(\tilde{t}_i)(t-\tau_1)\le \x_{i}(\tau_1)\quad\mbox{ for all } t\in(\tau_1,\tau_1+\delta_{i})\mbox{ and } \tilde{t}_{i}\in(\tau_1,t)\,,
\end{gather*}
and consequently
\[
 \x_{i+1}(t)-\x_{i}(t)\ge \x_{i+1}(\tau_1)-\x_{i}(\tau_1)=\frac{h}{M}\qquad\mbox{for all } t\in(\tau_1,\tau_1+\min\{\delta_{i},\delta_{i+1}\})\,,
\]
which is in contradiction with the existence of $\tau_2$.
\end{proof}

Before we proceed further, we state the following technical lemma, whose proof is postponed to Appendix~\ref{app:liploc_DPA}.

\begin{lemma}\label{lem:liploc_DPA}
    Let $\bar\rho$, $\beta$, $V$ and $W$ satisfy assumptions (In1), (A-$\beta$), (A-V) and (A-Wm) respectively. Further, let $h\in(0,1)$ and $\x\in\calC([0,T];\R^{N+1})$ be a solution of \eqref{eq:particle_complete} satisfying properties (i) and (ii) of Theorem~\ref{thm:existence-DPA}. Then there exists a constant $c_{\f}>0$, depending only on $V$, $W$, $m$ and $M$ such that
    \begin{align*}
        |\f_{i+1}-\f_i| &\le c_{\f} |K_i|\,,\qquad\text{and}\\
        |\f_{i+1}-2\f_i+\f_{i-1}| &\le c_{\f} \left(|K_i|^2 + |K_{i-1}|^2 + \bigl||K_i|-|K_{i-1}|\bigr|\right).
    \end{align*}
    
    In particular, if $W$ is the Newtonian potential, i.e.\ $W$ satisfies (A-Wn), then
    \[
        c_{\f} = \max\Bigl\{\|V''\|_{L^\infty(\R)} +2M,\Lip(V'')\Bigr\}\,,
    \]
    which depends only on $V$ and $M$.
\end{lemma}

In the case when $\bar\rho$ satisfies additionally (In2), we also obtain an upper bound for the increments of $\x$, which is essential in establishing the limiting gradient structure for \eqref{eq:main} in the general case.

\begin{lemma}\label{lem:upper-bound}
    Let $\bar\rho$, $\beta$, $V$ and $W$ satisfy assumptions (In1), (A-$\beta$), (A-V) and (A-Wm) respectively. Then for any $h\in(0,1)$, a solution $\x\in\calC([0,T];\R^{N+1})$ of \eqref{eq:particle_complete} satisfies
    \[
        |K_i|(t) = \x_{i+1}(t)-\x_i(t) \leq e^{\mu T}\max_{j=0,\ldots,N-1}|\bar\x_{j+1}- \bar\x_j|\,,\qquad \text{$i=0,\ldots,N-1$, $t \in [0,T]$}\,,
    \]
for any constant $\mu>c_{\f}\beta_{max}$, where $c_{\f}$ is the constant appearing in Lemma~\ref{lem:liploc_DPA}.

In particular, if $\bar\rho$ satisfies additionally (In2), then
\[
    |K_i|(t) \le e^{\mu T} \frac{h}{\sigma}\,, \qquad \text{$i=0,\ldots,N-1$, $t \in [0,T]$}\,.
\]
\end{lemma}
\begin{proof} Let $c_K:=\max_{j=0,\ldots,N-1}|\bar\x_{j+1}-\bar\x_j|$. We introduce the time
\[ 
    \tau_1:= \inf \left\{ s \in (0,T] :\,\exists\,i\;\text{such that}\;|K_i|(s) \geq c_K e^{\mu s} \right\},
\]
Note that if $\tau_1=T$ the bound follows trivially, so we consider $\tau_1 < T$. For sake of contradiction, assume that there is a time $\tau_2 \in (\tau_1,T]$ such that  
\begin{equation}\label{contrprinmin}
|K_i|(t) > c_K e^{\mu t} \quad \mbox{ for every $t \in (\tau_1,\tau_2]$.}
\end{equation}
By construction, at the time $\tau_1$ we have that 
\[ 
    |K_i|(\tau_1) = c_Ke^{\mu \tau_1}\quad \mbox{ and }\quad |K_j|(\tau_1) \leq c_Ke^{\mu \tau_1}\qquad \forall\,j\neq i,
\]
thus, in particular,
\[ \rho_{i+1}^h(\tau_1) \geq \rho_i^h(\tau_1)\quad \mbox{ and }\quad \rho_{i-1}^h(\tau_1) \geq \rho_i^h(\tau_1)\,. \]
We then compute
\begin{align*}
\frac{d}{dt}\left[ e^{-\mu t}|K_i|(t)\right] \bigg|_{t=\tau_1} &= e^{-\mu\tau_1}\left[(\dot{\x}_{i+1}(\tau_1) -\dot{\x}_i(\tau_1)) - \mu |K_i|(\tau_1)\right] \\
&= e^{-\mu\tau_1}\biggl[-\underbrace{\left(\beta(\rho_{i+1}^h(\tau_1))-\beta(\rho_i^h(\tau_1)\right)\f_{i+1}^{-}}_{\ge 0}+\underbrace{\left(\beta(\rho_{i-1}^h(\tau_1))-\beta(\rho_i^h(\tau_1)\right)\f_i^{+}}_{\le 0} \\
&\hspace{6em} \qquad-\beta(\rho_i^h(\tau_1))\left(\f_{i+1}-\f_i\right)- c_K e^{\mu\tau_1}\mu  \biggr] \\
&\leq e^{-\mu\tau_1}\left[\beta(\rho_i^h(\tau_1))|\f_{i+1}-\f_i|-c_K e^{\mu\tau_1}\mu \right]\,.
\end{align*}
Under the assumptions on $V$ and $W$, Lemma~\ref{lem:liploc_DPA} gives an appropriate constant $c_{\f}>0$ such that
\[
 |\f_{i+1}-\f_i| \le c_{\f}|K_i|\,.
\]
This yields,
\begin{align*}
\frac{d}{dt}\left[ e^{-\mu t}|K_i|(t))\right]_{|_{t=\tau_1}}&\leq e^{-\mu \tau_1}\left [\beta_{max}c_{\f}\,|K_i|(\tau_1)- c_K e^{\mu\tau_1}\mu  \right]\\
& \leq c_K\left [\beta_{max}c_{\f}-\mu  \right]<0\,,
\end{align*}
where the last inequality holds because of our choice of $\mu$. The inequality above leads to a contradiction. Indeed, we find some positive $\delta \ll 1$ for which $\tau_1 + \delta < \tau_2$ and $\frac{d}{dt}\left[ e^{-\mu t}|K_i(t)|\right] < 0$ for all $t \in (\tau_1,\tau_1+\delta]$. Then, 
\[ e^{-\mu t}|K_i(t)| = c_K + \int_{\tau_1}^t \frac{d}{ds}\left[ e^{-\mu s}|K_i(s)|\right]ds \leq c_K\,,  \]
which clearly contradicts~\eqref{contrprinmin}, thereby concluding the of the first statement.

If $\bar\rho$ satisfies additionally (In2), then by construction,
\[
    h = \int_{\bar\x_i}^{\bar\x_{i+1}} \bar\rho(y)\,dy \ge \sigma\,|\bar\x_{i+1}-\bar\x_{i}| \qquad\text{for any $i=0,\ldots,N-1$}\,,
\]
which consequently gives the bound $c_K \le h/\sigma$.
\end{proof}

\subsection{Existence result for DPA}

To establish the existence of solutions to \eqref{eq:particle_complete}, we begin by discussing the solvability of an auxiliary system. Namely, for each $\varepsilon>0$, we consider the system
\begin{equation}\label{eq:particle_aux}\tag{{DPA}$^\varepsilon$}
\left\{\qquad \begin{aligned}
\dot{\x}_0 &= -\beta(\rho_0^{\varepsilon,h})\f_0^{\varepsilon,-} - \beta_{max}\f_0^{\varepsilon,+}\\
    \dot\x_i &= -\beta(\rho_i^{\varepsilon,h})\f_i^{\varepsilon,-}-\beta(\rho_{i-1}^{\varepsilon,h})\f_i^{\varepsilon,+}\\
    \dot{\x}_N &= -\beta_{max}\f_N^{\varepsilon,-}-\beta(\rho_{N-1}^{\varepsilon,h})\f_N^{\varepsilon,+}
    \end{aligned}\qquad\text{for $i=1,\ldots,N-1$}\,, \right.
\end{equation}
where we replace $\f_i$ with
\[
    \f_i^\varepsilon = V'(\x_i) - \sum_{j>i} h W'\bigl(\max\{\varepsilon, \x_j - \x_i\}\bigr) + \sum_{j<i} h W'\bigl(\max\{\varepsilon, \x_i - \x_j\}\bigr)\,,
\]
and $\rho^h$ with
\[
    \rho_i^{\varepsilon,h}(t) = \frac{h}{\max\{\varepsilon,\x_{i+1}(t) -\x_i(t)\}},\qquad i=0,\ldots,N-1\,.
\]
In this way, the right-hand side of \eqref{eq:particle_aux} for each $\varepsilon>0$ is easily seen to be continuous w.r.t.\ $\x$. In particular, the classical Peano theorem for local-in-time existence applies.

\medskip

With the uniform estimate obtained in Lemma~\ref{lem:lower_DPA}, we are now prepared to proof the existence of local solutions to \eqref{eq:particle_complete}. 

\begin{lemma}\label{lem:local-DPA}
    Let $\bar\rho$, $\beta$, $V$ and $W$ satisfy the assumptions (In1), (A-$\beta$), (A-V) and (A-Wm) respectively. Then there exists a local-in-time solution $\x$ to the deterministic particle approximation \eqref{eq:particle_complete}.
\end{lemma}
\begin{proof}
    Due to Lemma~\ref{lem:lower_DPA}, any solution $\x^\varepsilon$ of \eqref{eq:particle_aux} for $\varepsilon < h/M$, is in fact a solution of \eqref{eq:particle_complete}. 
\end{proof}

With a local solution of \eqref{eq:particle_complete} at hand, we now show that the solution can be extended globally.

\begin{lemma}\label{lem:extend-DPA}
    Let $\bar\rho$, $\beta$, $V$ and $W$ satisfy the assumptions (In1), (A-$\beta$), (A-V) and (A-Wm) respectively. Further, let $\x$ be a local solution to \eqref{eq:particle_complete}. Then
    \begin{gather*}
        \sup\nolimits_{i=0,\ldots,N} |\x_i(t)| \le q_1(t),\qquad |\x_N(t)-\x_0(t)|\le q_2(t)\,,\quad\text{and}\\[0.5em]
        \sup\nolimits_{i=0,\ldots,N}|\f_i| \le C\bigl(1 + q_1(t) + q_2(t)\bigr)\,,
    \end{gather*}
    for some constant $C>0$ independent of $N$, and where $q_1,q_2\in C([0,\infty))$ are monotonically increasing positive functions. In particular, the solution $\x$ may be extended to any bounded interval $[0,T]$, $T>0$.
\end{lemma}
\begin{proof}
    Since $V''\in L^\infty(\R)$ and $W'$ satisfies the linear growth assumption, we have that
    \begin{align*}
        |\f_i| \le |V'(0)| + \|V''\|_{L^\infty(\R)}|\x_i| + m c_W\bigl(1 +|\x_N-\x_0|\bigr)\,,\qquad i=0,\ldots,N\,.
    \end{align*}
Therefore, a local solution $\x$ of \eqref{eq:particle_complete} satisfies
    \begin{align*}
    \frac{1}{2}\frac{d}{dt}|\x_i|^2 &= -\langle \x_i, \beta(\rho_i^h)\f_i^-+\beta(\rho_{i-1}^h)\f_i^+\rangle
    \le \beta_{max}|\x_i||\f_i| \\
    &\le c_1 |\x_i|^2 + c_2|\x_N-\x_0| + c_3\,,\qquad i=0,\ldots,N\,,
\end{align*}
with appropriate constants $c_1,c_2,c_3>0$.

On the other hand, due to Lemma~\ref{lem:liploc_DPA}, we find that
\begin{align*}
     \frac{1}{2}\frac{d}{dt}|\x_N-\x_0|^2 &= -\langle \x_N-\x_0,\beta_{max}\f_N^- + \beta(\rho_{N-1}^h)\f_N^+ -\beta(\rho_0^{h})\f_0^- - \beta_{max}\f_0^+\rangle \\
     &= -\langle \x_N-\x_0,\beta_{max}(\f_N -\f_0)+ (\beta(\rho_{N-1}^h)- \beta_{max})\f_N^++ (\beta_{max} - \beta(\rho_0^{h}))\f_0^- \rangle 
     \\
        &\le c_{\f}\beta_{max} |\x_N-\x_0|^2 + \beta_{max}|\x_N-\x_0|\bigl(|\f_N| + |\f_0|\bigr) \\
        &\le c_4|x_N-\x_0|^2 + c_5\Bigl( |\x_0|^2 + |\x_N|^2\Bigr) + c_6\,,
\end{align*}
with appropriate constants $c_4,c_5,c_6>0$.

Considering the functional
\[
    \mathcal{J}(\x):= |\x_0|^2 + |\x_N-\x_0|^2 + |\x_N|^2\,,
\]
we deduce the differential inequality
\[
    \frac{d}{dt}\mathcal{J}(\x) \le \tilde{c}_1 \mathcal{J}(\x) + \tilde{c_2}\,.
\]
with constants $\tilde c_1,\tilde c_2>0$. In particular, an application of Gr\"onwall's inequality yields
\[
    \mathcal{J}(\x(t)) \le \Bigl(\mathcal{J}(\bar\x) + \tilde c_2 t\Bigr)\, e^{\tilde c_1 t}\qquad\text{for every $t\in[0,T]$}\,,
\]
which then allows us to extend our local solutions to global ones. 
\end{proof}


\subsection{A gradient flow structure for DPA} \label{sec:discrete-GF}

In this section, we propose a gradient flow structure for the particle system \eqref{eq:particle_complete}, i.e.\ we postulate a {\em free energy} $\calF_h$ and  {\em dual dissipation potential} $\calR_h^*$ for which equation \eqref{eq:particle_complete} takes the form
\begin{align}\label{eq:GGF-discrete}
	\dot\x = \partial_2\calR_h^*(\x,-\calF_h'(\x))\,,
\end{align}
where $\calF_h'$ is the Fr\'echet derivative of the discrete free energy $\calF_h:\calK_N\to \R$ given by
\begin{equation}\label{discrete_ener_fun}
	\calF_h(\x) = \sum_{i=0}^{N-1} V(\x_i) + \frac{1}{2}\sum_{i=0}^{N-1}\sum_{j\ne i} h W(\x_i-\x_j)\,,\qquad \x\in\calK_N.
\end{equation}
Here, the set $\calK_N\subset\R^N$ denotes the convex cone
\[
	\calK_N := \biggl\{ \x \in\R^N\,:\, x_{i}\le x_{i+1}\quad\text{for all\, $i=0,\ldots,N-1$}\biggr\}\,.
\]
For $\x\in \intK_N$, the Gate\'aux derivative of $\calF_h$ in $\x$ along $\eta\in\R^N$ reads
\[
	D\calF_h(\x)[\eta] = \sum_{i=0}^{N-1} \eta_i \left(  V'(\x_i) + \sum_{j\ne i} h W'(\x_i-\x_j)\right) =: \langle \eta,\sfF^h(\x)\rangle\,.
\]
In particular, if $\x=\x(t)\in \intK_N$ satisfies \eqref{eq:particle_complete} with $\f = \sfF^h(\x)$ for all $t\ge 0$, we then find
\begin{align}\label{eq:discrete-chain-rule}
	\begin{aligned}
	\frac{d}{dt}\calF_h(\x) &= \langle \dot\x,\sfF^h(\x)\rangle = -\sum_{i=0}^{N-1} \bigl[\beta(\rho_i^h)\f_i^- + \beta(\rho_{i-1}^h)\f_i^+ \bigr] \f_i
	= -\calD_h(\x(t))\,,
	\end{aligned}
\end{align}
where $\calD_h$ is the {\em discrete dissipation functional} given by
\begin{align*}
    \calD_h(\x) := \sum_{i=0}^{N-1} \Bigl[\beta(\rho_i^h) (\f_i^-)^2 + \beta(\rho_{i-1}^h)(\f_i^+)^2\Bigr]\,,
\end{align*}
with the convention that $\rho_{-1}^h = 0$.

\subsubsection*{Dual dissipation potential}

We now postulate a dual dissipation potential $\calR_h^*$ defined by
\begin{align}\label{eq:discrete-dual-dissipation}
	\calK_N\times\R^N\ni (\x,\zeta)\mapsto \calR_h^*(\x,\zeta) := \frac{1}{2}\sum_{i=0}^N \beta(\rho_{i-1}^h)(\zeta_i^-)^2 + \beta(\rho_i^h)(\zeta_i^+)^2\in [0,\infty)\,.
\end{align}
Then for all $\eta\in\R^N$,
\[
	D_2\calR_h^*(\x,\zeta)[\eta] 
	= \sum_{i=0}^N \eta_i\bigl( \beta(\rho_{i-1}^h)\zeta_i^- + \beta(\rho_i^h)\zeta_i^+\bigr) = \langle \eta,\partial_2\calR_h^*(\x,\zeta)\rangle\,.
\]
Recalling that $\f = \sfF_h(\x)$, we compute
\[
	(\partial_2\calR_h^*(\x,-\sfF^h(\x)))_i = \beta(\rho_{i-1}^h)(-\f_i)^- + \beta(\rho_i^h)(-\f_i)^+ = -\beta(\rho_{i-1}^h)\f_i^+ -  \beta(\rho_i^h)\f_i^-= \dot\x_i\,,
\]
i.e.\ the pair $(\calF_h,\calR_h^*)$ brings \eqref{eq:particle_complete} into the gradient form \eqref{eq:GGF-discrete}.

\subsubsection*{Discrete Energy-Dissipation balance}

Let $(\x,\j) \in\calC([0,T];\intK_N)\times L^1((0,T);\R^N)$ be a pair satisfying the discrete continuity equation
\begin{align}\label{eq:cont-eq-discrete}
	\dot \x(t) = \j(t)\qquad\text{for almost every $t\in(0,T)$}\,. \tag{{CE}$_h$}
\end{align}
From equality $\dot\x = \partial_2\calR_h^*(\x,-\sfF_h(\x))$, we obtain the flux-force relation
\begin{align}\label{eq:discrete-force-flux}\tag{{FF}$_h$}
		\j =  \partial_2\calR_h^*(\x,-\sfF_h(\x))\quad\Longleftrightarrow\quad \calR_h(\x,\j) + \calR_h^*(\x,-\sfF^h(\x)) = \langle -\sfF^h(\x),\j\rangle\,,
\end{align}
where the {\em dissipation potential} $\calR_h$ is the Legendre dual of $\zeta\mapsto\calR_h^*(\x,\zeta)$ given by
\begin{align}\label{eq:discrete-dissipation-potential}
	\calR_h(\x,\j) = \frac{1}{2}\sum_i \left[ \beta(\rho_{i-1}^h)\left(\frac{\j_i^-}{\beta(\rho_{i-1}^h)}\right)^2 + \beta(\rho_i^h)\left(\frac{\j_i^+}{\beta(\rho_i^h)}\right)^2 \right].
\end{align}
From the variational force-flux relation \eqref{eq:discrete-force-flux} and the chain rule \eqref{eq:discrete-chain-rule}, we then obtain the discrete Energy-Dissipation balance 
\begin{align}\label{eq:discrete-EDB}\tag{{EDB}$_h$}
    \int_s^t \calR_h(\x(r),\j(r)) + \calR_h^*(\x(r),-\sfF^h(\x(r)))\,dr = \calF_h(\x(s)) - \calF_h(\x(t))
\end{align}
for any interval $(s,t)\subset[0,T]$.


\subsection{DPA as a numerical scheme}\label{sec:num} As highlighted in previous works, see \cite{diFFRa2019,FRa2018} and references therein, beyond the motivation of introducing a DPA as an analytical tool, there is also its fruitfully use as a numerical scheme for the continuous (limit) density counterpart solution to \eqref{eq:main}. 

More precisely, given an initial datum $\bar{\rho}$, we built a set of well-ordered initial particles $\{\bar{\x}_i\}_{i=0,\ldots,N}$ according to \eqref{eq:dscr_IC_un}, and we solve the corresponding ODE system \eqref{eq:particle_complete}. We then reconstruct in each \emph{cell} $K_i$ a density $\rho_i^h$ according to \eqref{eq:disc_density}---the following section is devoted to the rigorous validation on this approach. Before entering into the details, we want to show a collection of numerical results that underline some key aspects in the arguments used, such as the evolution of the $BV$-norm in contrast to the $H^1$-norm, and the discrete energy-dissipation balance \eqref{eq:discrete-EDB}. 

In each of the examples below, we will plot the evolution in time of the discrete density  \eqref{eq:disc_density}, or more precisely the piecewise constant interpolation \eqref{eq:piec_dens} defined in Section \ref{sec:Continuous}, the evolution in time of the $BV$-norm compared with the $H^1$-norm, and the discrete energy-dissipation identity. In all the examples we consider as initial condition
\[
\bar{\rho}(x)=\frac{3}{4}(1-x^2)_{+},
\]
and the discrete set composed of $N=200$ particles. In Figure \ref{fig:test_1}, we consider an attractive Newtonian interaction potential $W(x)=|x|$ and $V(x)=0$, while in Figures \ref{fig:test_2} and \ref{fig:test_3}, we show the evolution under the influence of a repulsive Newtonian interaction potential $W(x)=-|x|$ in combination with a confining external potential $V(x)=x^2/2$ and $V(x)=0$ respectively.

Figure~\ref{fig:test_1} clearly indicates that the $BV$-norm behaves better than the $H^1$-norm. In fact, the $H^1$-norm of the approximation in this scenario is expected to explode as $t\to \infty$, since its stationary solution is (up to a constant multiple) an indicator function on the interval $(-1/2,1/2)$. Furthermore, all figures verify the validity of the discrete energy-dissipation balance \eqref{eq:discrete-EDB}, as suggested analytically, and is something we observe all across the levels of discretization.

\begin{figure}[htbp]
  \centering
    \includegraphics[width=5cm,height=4cm]{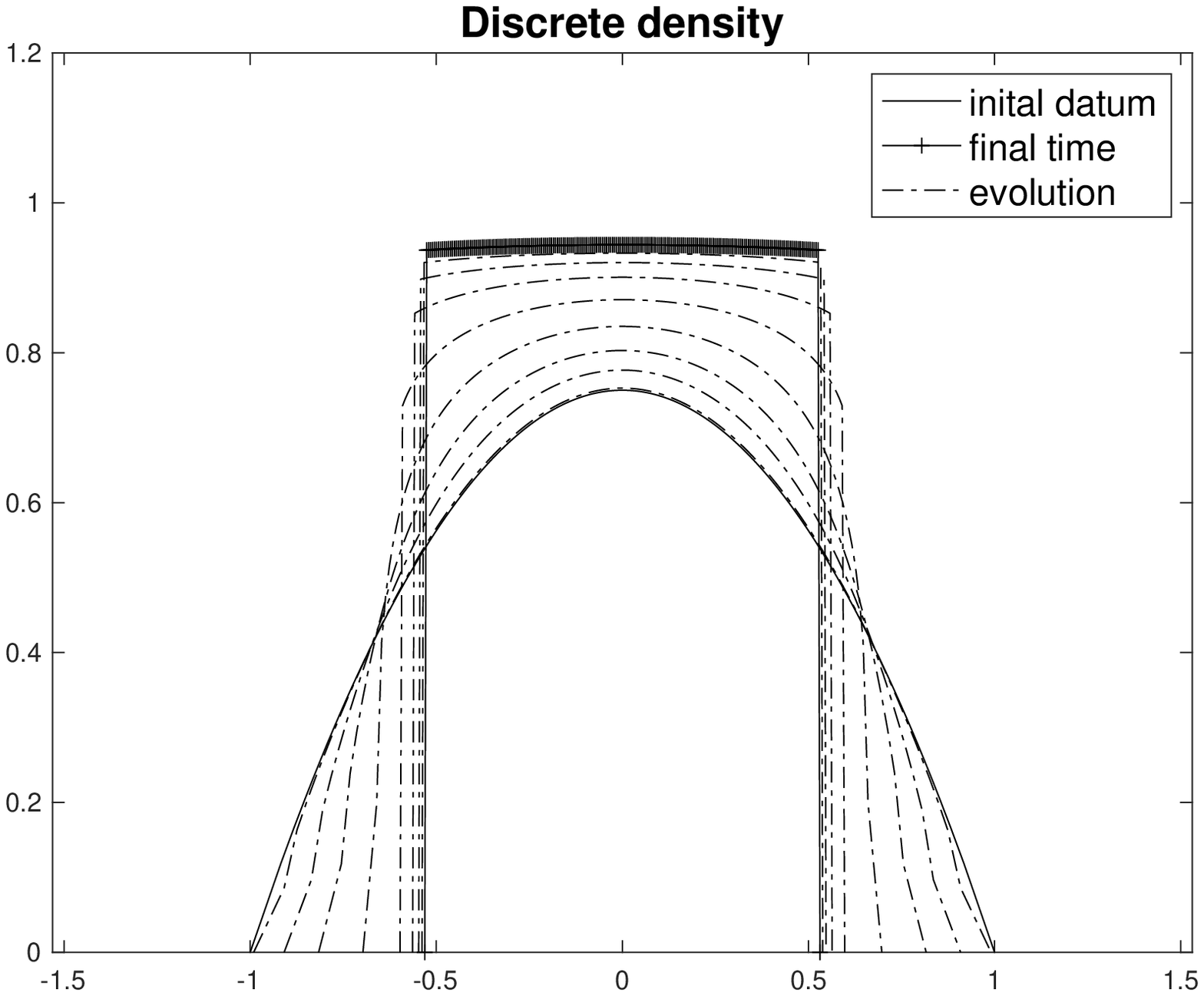}
    \includegraphics[width=5cm,height=4cm]{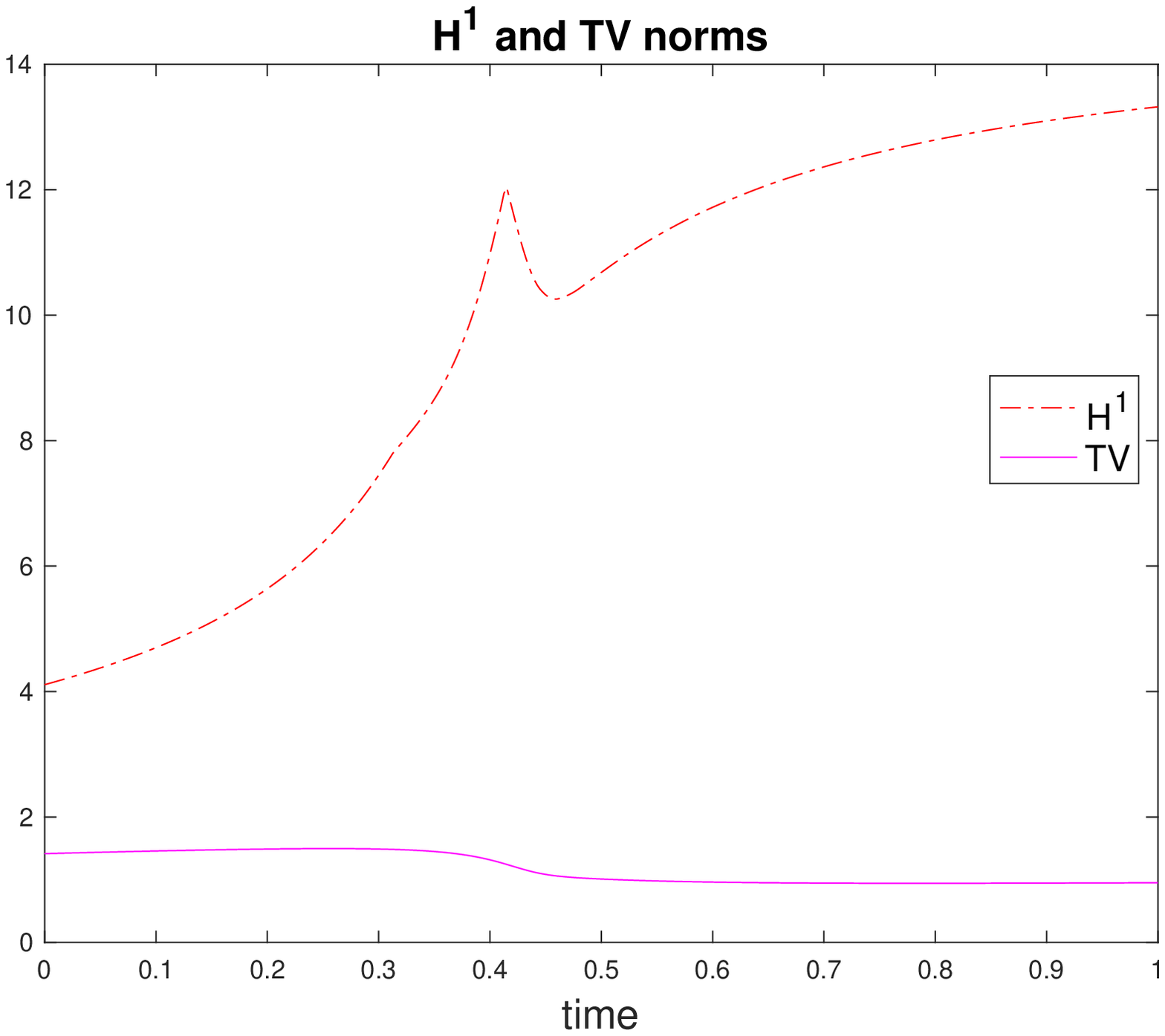} 
    \includegraphics[width=5cm,height=4cm]{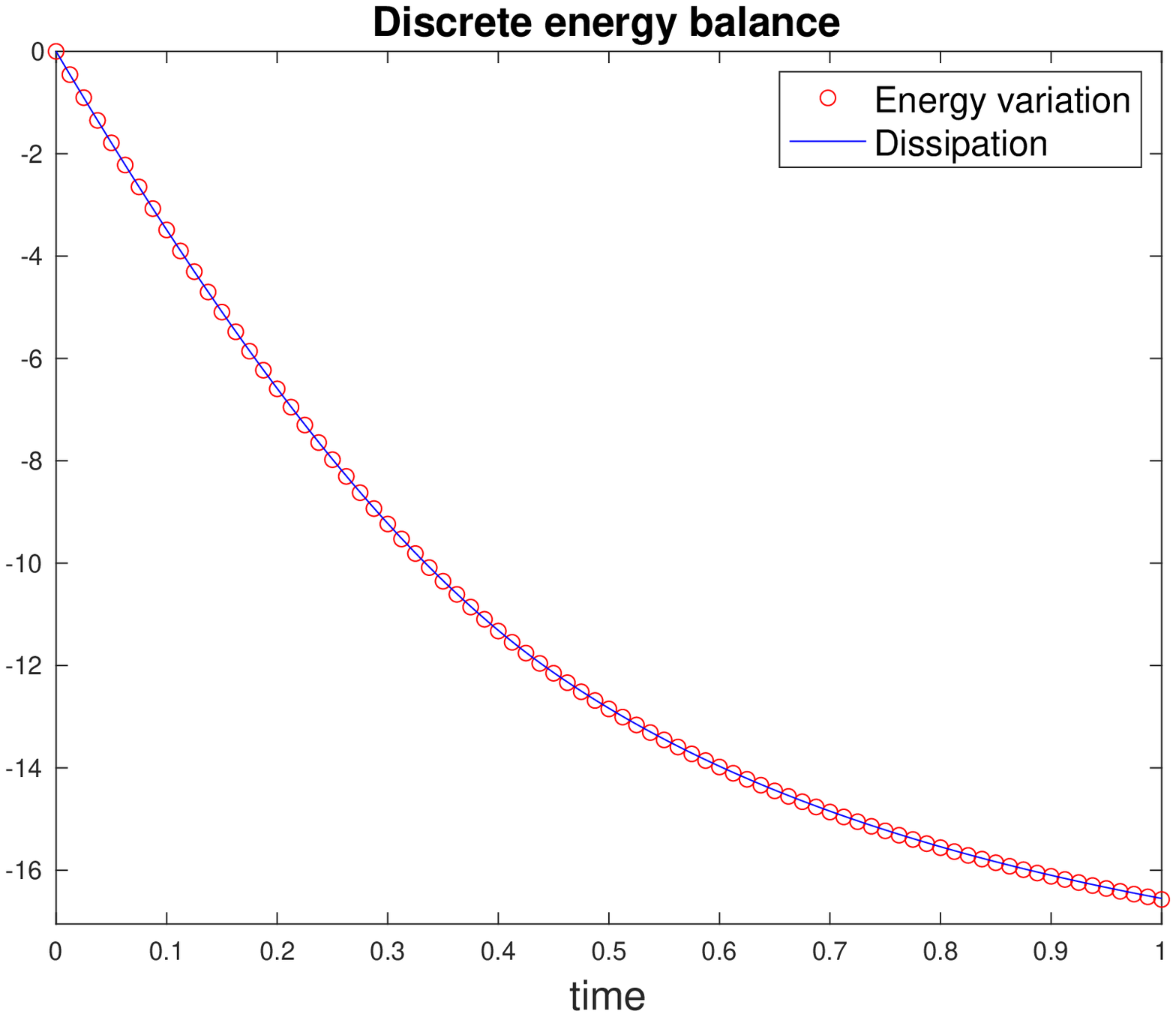}
\caption{Under the influence of an attractive Newtonian interaction potential $W(x)=|x|$, we plot: The evolution of the reconstructed discrete density from \eqref{eq:particle_complete}  (left),  evolution of the $BV$ and $H^1$ norms (center) and the discrete energy-dissipation identitiy (right). This example highlights the convenience of handling the $BV$ norm instead of the $H^1$.}
\label{fig:test_1}
\end{figure}
\begin{figure}[htbp]
  \centering
    \includegraphics[width=5cm,height=4cm]{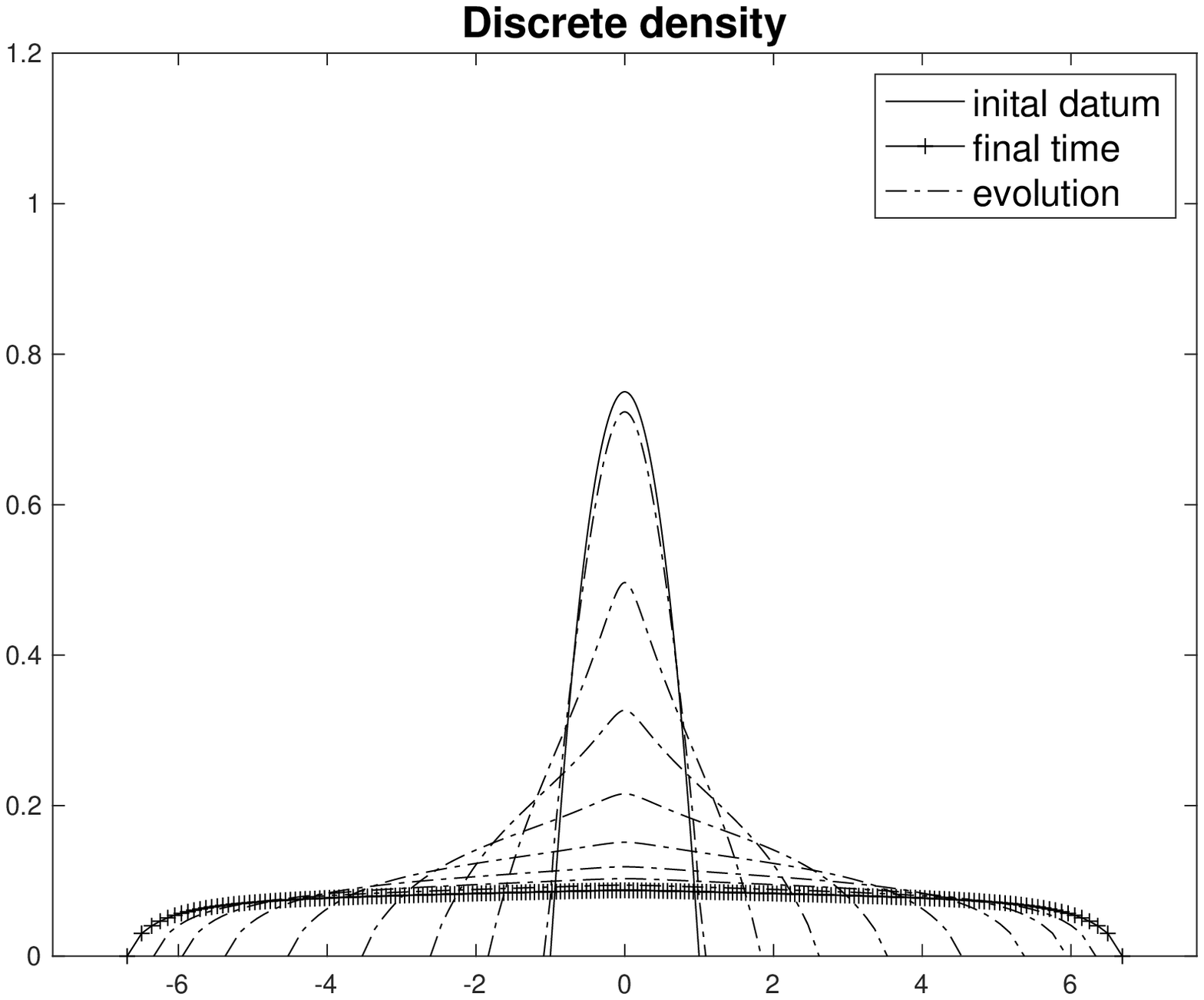}
    \includegraphics[width=5cm,height=4cm]{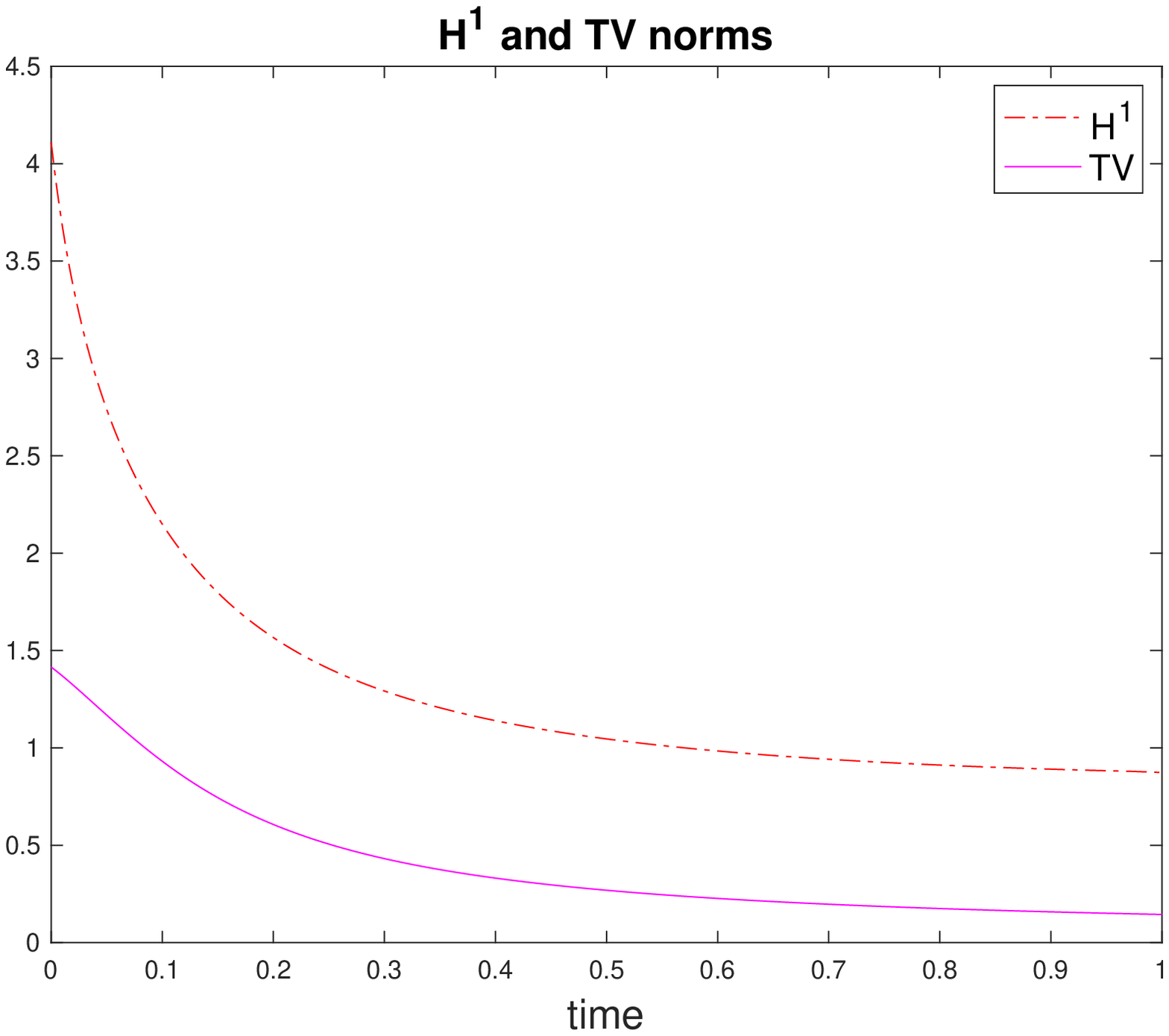}
    \includegraphics[width=5cm,height=4cm]{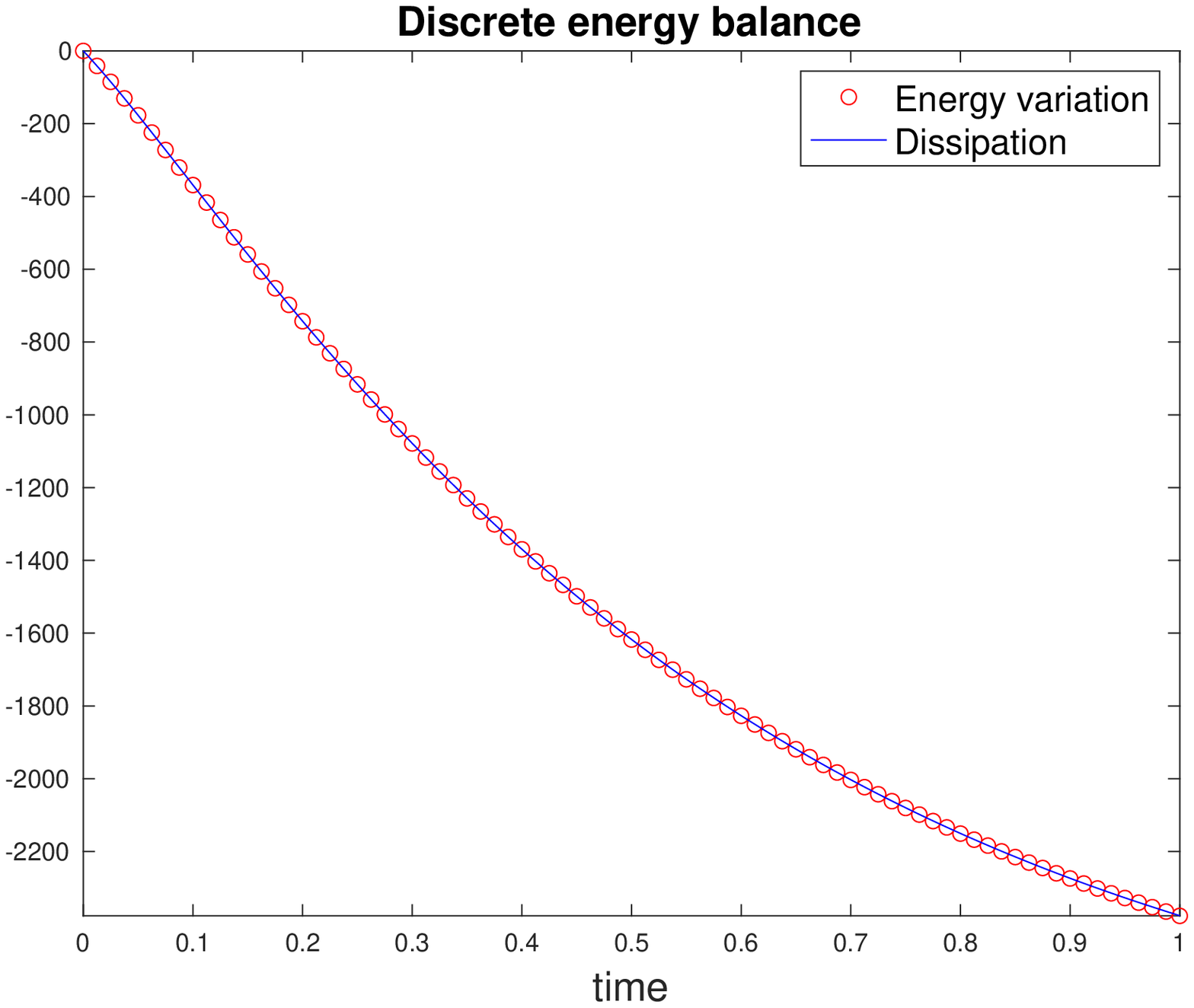}
\caption{Under the influence of a repulsive Newtonian interaction potential $W(x)=-|x|$ and a confining external potential $V(x)=x^2/2$, we plot:  evolution of the reconstructed discrete density from \eqref{eq:particle_complete}  (left),  evolution of the $BV$ and $H^1$ norms (center) and the discrete energy-dissipation identitiy (right).}
\label{fig:test_2}
\end{figure}
\begin{figure}[htbp]
  \centering
    \includegraphics[width=5cm,height=4cm]{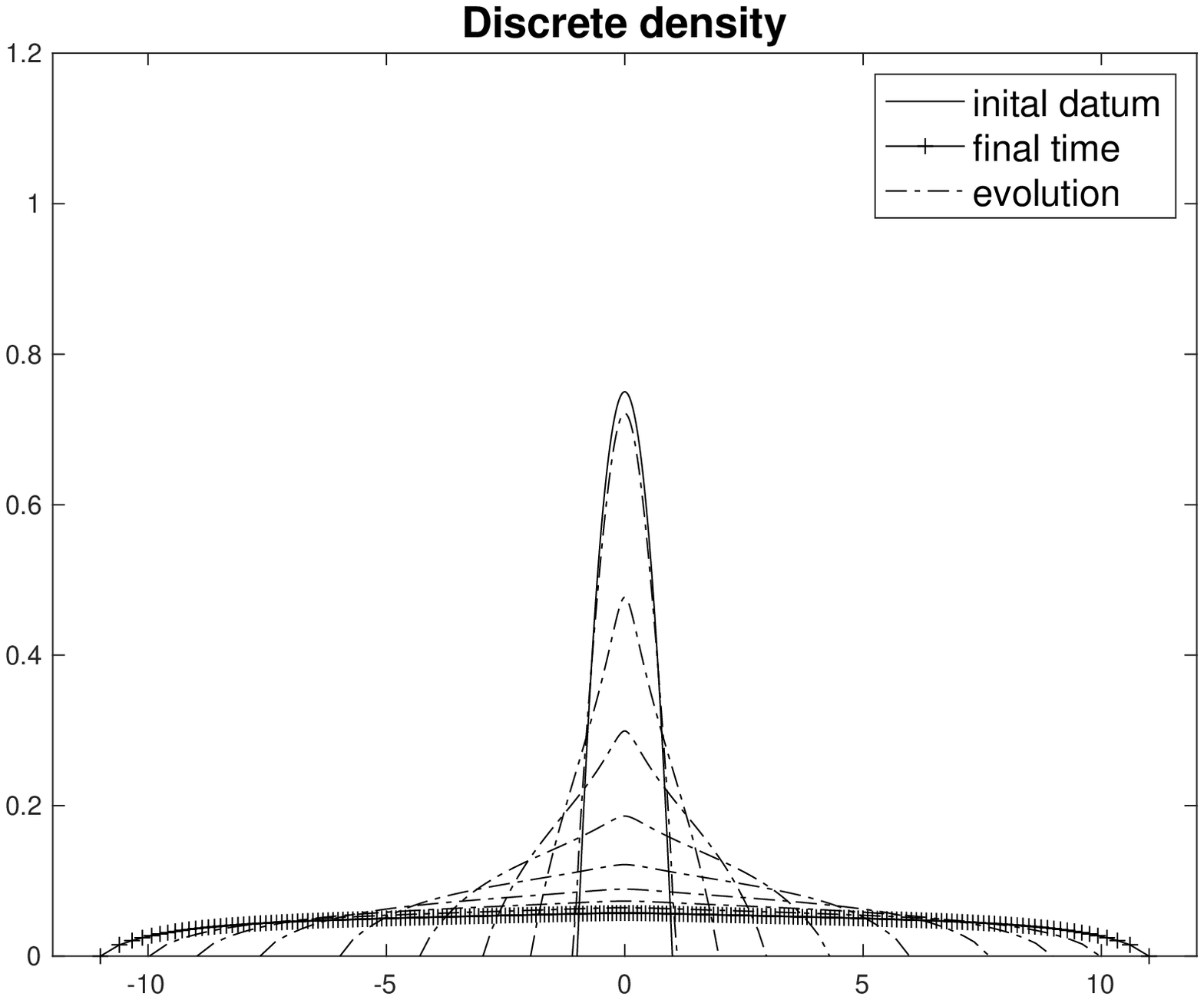}
    \includegraphics[width=5cm,height=4cm]{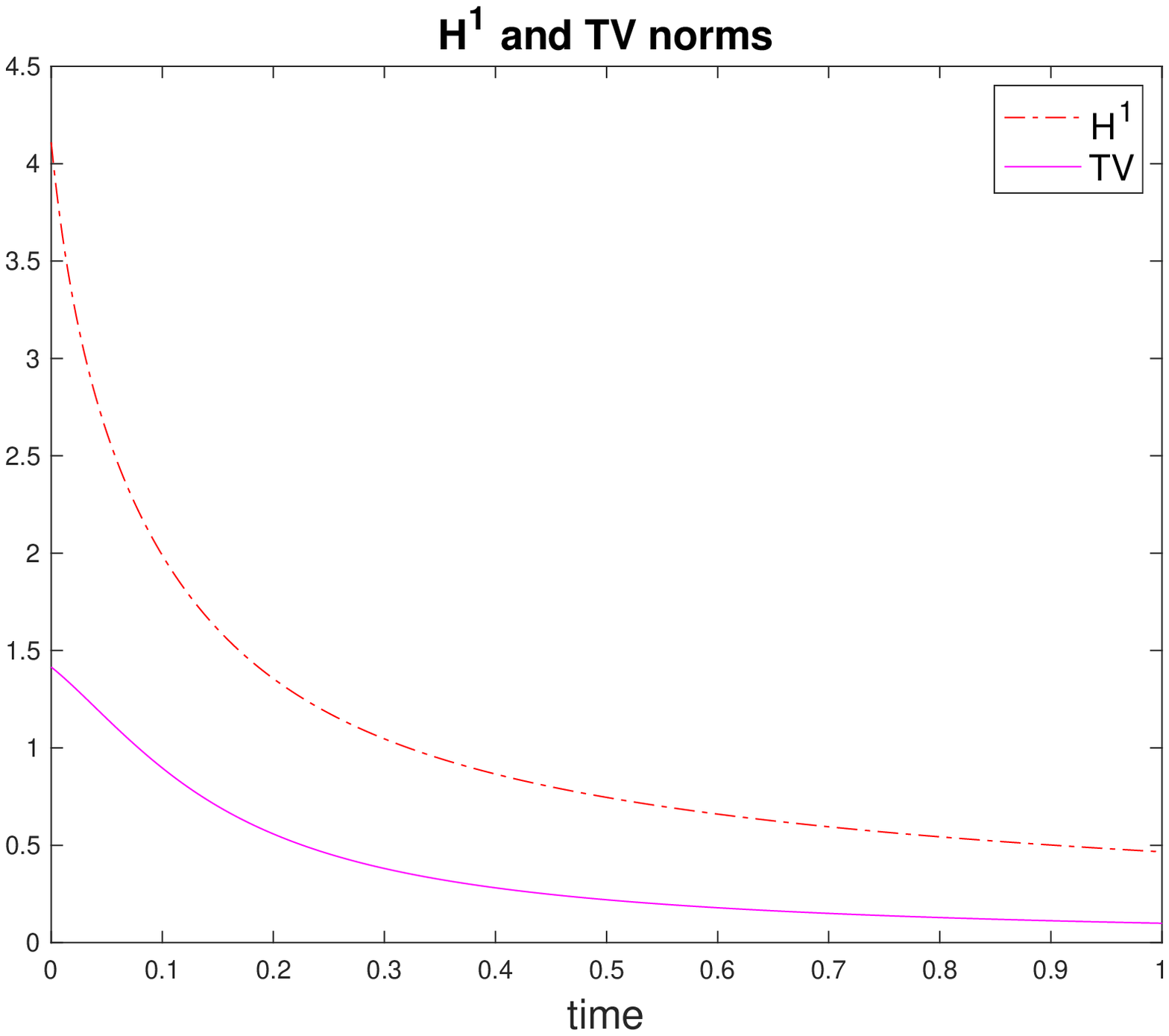}
    \includegraphics[width=5cm,height=4cm]{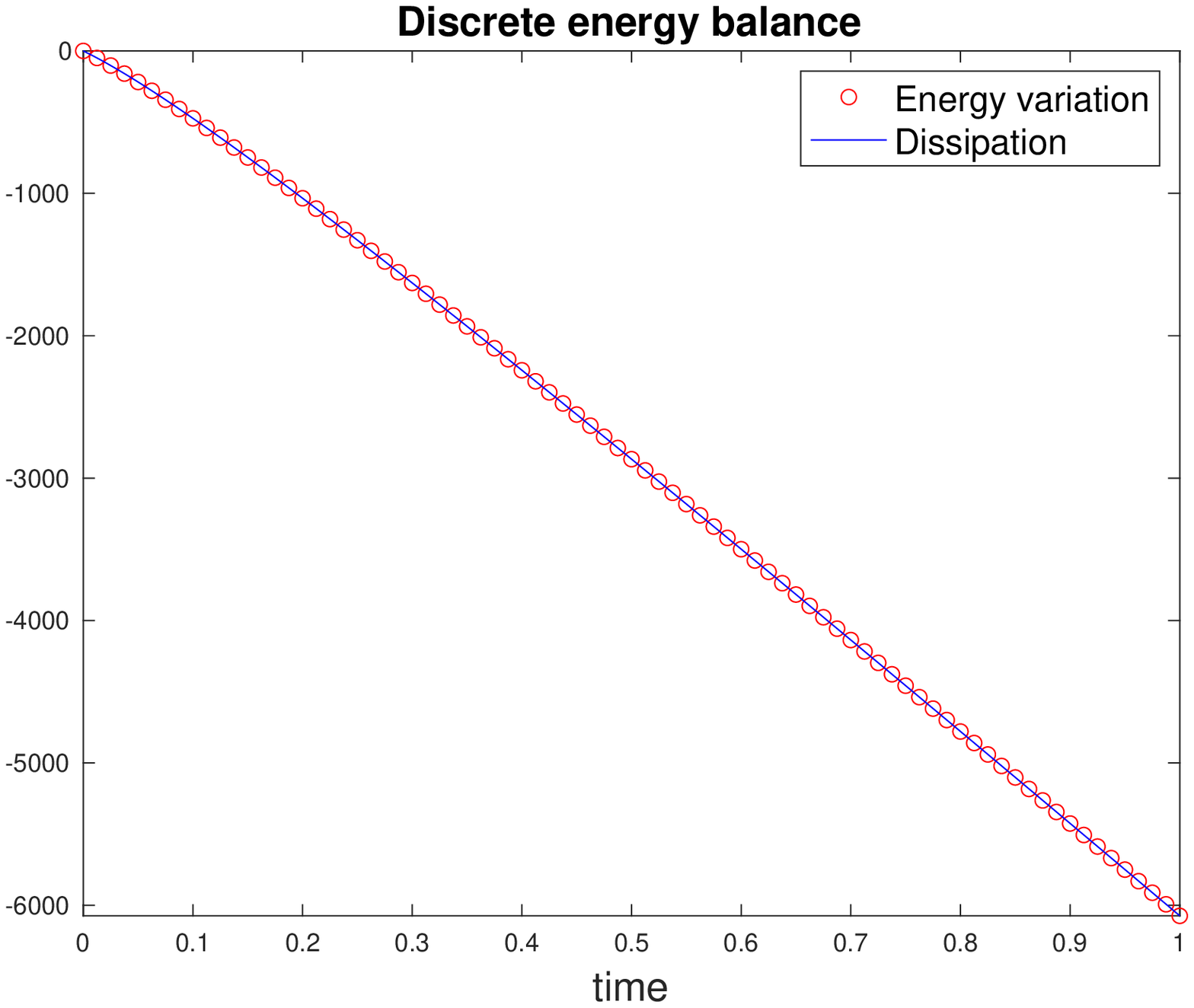}
\caption{Under the influence of a repulsive Newtonian interaction potential $W(x)=-|x|$, we plot:  evolution of the reconstructed discrete density from \eqref{eq:particle_complete}  (left),  evolution of the $BV$ and $H^1$ norms (center) and the discrete energy-dissipation identitiy (right).}
\label{fig:test_3}
\end{figure}

\section{Continuous reconstruction and compactness} \label{sec:Continuous}

Let $\x\in C([0,T];\R^{N+1})$ be a solution of the deterministic particle approximation \eqref{eq:particle_complete}. We define the piecewise constant density reconstruction as
\begin{equation}\label{eq:piec_dens}
        \hat\rho_t^h(x) := \sum_{i=0}^{N-1} \rho_i^h(t) \bbmI_{K_i(t)}(x)\,,\qquad \rho_i^h = \frac{h}{|K_i|},\qquad |K_i| = \x_{i+1} - \x_i\,.
\end{equation}
We can associate to \eqref{eq:piec_dens} the following flux reconstruction 
\begin{equation}\label{eq:piec_flux}
    \hat \jmath_t^h(x) := \sum_{i=0}^{N-1} \rho_i^h(t)\, u_i^h(t,x) \bbmI_{K_i(t)}(x)\,,
\end{equation}
where the velocity field $u_i^h=u_i^h(t,x)$ is defined by
\begin{equation}\label{eq:piec_vel}
u_i^h(t,x) := \frac{\x_{i+1}(t)-x}{|K_i|(t)}\dot \x_i(t) + \frac{x-\x_i(t)}{|K_i|(t)}\dot \x_{i+1}(t)\qquad \text{for $x\in K_i$,\; $i=0,\ldots,N-1$\,.} 
\end{equation}

Observe that this specific reconstruction is chosen such that the pair $(\hat\rho^h,\hat\jmath^h)$ satisfies the continuity equation mentioned in the introduction.

\begin{lemma}\label{lem:reconstruction-continuity}
	Let $h\in(0,1)$ and $\{\x^h\}$ be a family of solutions to \eqref{eq:particle_complete} with its corresponding reconstructed family of pairs $\{(\hat \rho^h,\hat\jmath^h)\}_{h\in(0,1)}$. The following hold true:
	\begin{enumerate}[label=(\roman*)]
	    \item For each $h\in(0,1)$, the pair $(\hat \rho^h,\hat\jmath^h)\in \mathcal{CE}(0,T)$ in the sense of Definition~\ref{def:cont-eq};
	    \item The family $\{t\mapsto |\hat\jmath_t^h|(\R)\}_{h\in(0,1)}$ is uniformly integrable.
	\end{enumerate}
\end{lemma}
\begin{proof}
    The weak-$*$ continuity of the curve $[0,T]\ni t\mapsto\hat\rho_t^h\in\calM^+(\R)$ follows from the continuity of $t\mapsto \x_i(t)$ for every $i=1,\ldots,N$, while the measurability of the family $\{\hat\jmath_t^h\}_{t\in[0,T]}\subset \calM(\R)$ follows from the measurability of $t\mapsto \rho_i^h(t)$ and $t\mapsto \dot\x_i(t)$ for every $i=1,\ldots,N$. 
    
    From Lemma~\ref{lem:extend-DPA}, we easily deduce that
    \[
        \sup_{h\in(0,1)}\sup_{t\in[0,T]}\|\dot \x(t)\|_{\ell^\infty} <\infty\,,
    \]
    and hence also
    \[
        c_u:=\sup_{h\in(0,1)}\sup_{t\in[0,T]}\max_{i=0,\ldots,N-1} \|u_i^h(t,\cdot)\|_{L^\infty(K_i)}<\infty\,.
    \]
    By definition, we then obtain for any Borel set $A\subset[0,T]$
    \[
        \sup_{h\in(0,1)}\int_A |\hat\jmath_t^h|(\R)\,dt \le \sup_{h\in(0,1)}\sum_{i=0}^{N-1} h \int_A \|u_i^h(t,\cdot)\|_{L^\infty(K_i)}\,dt \le c_u m |A|\,,
    \]
which proves the uniform integrability of the family $\{t\mapsto|\hat\jmath_t^h|(\R)\}_{h\in(0,1)}$. In particular, $\int_0^T |\hat\jmath_t^h|(\R)\,dt <\infty$ uniformly for all $h\in(0,1)$.

	Furthermore, we see that for any $\varphi\in\Lip_b(\R)$, it holds
	\begin{align*}
		\langle \partial_x\varphi,\hat\jmath_t^h\rangle &= \sum_{i=0}^{N-1}\rho_i^h(t)\left(\frac{\dot{\x}_i}{|K_i|(t)}\int_{K_i(t)}\varphi'(x)(\x_{i+1}(t)-x)dx+\frac{\dot{\x}_{i+1}}{|K_i|(t)}\int_{K_i(t)}\varphi'(x)(x-\x_{i}(t))dx\right) \\
		&= \sum_{i=0}^{N-1}  \frac{d}{dt}\left( h\intbar_{K_i(t)}\varphi(x)\,dx\right) = \frac{d}{dt}\int_\R \varphi(x)\,\hat\rho_t^h(x)\,dx\,.
	\end{align*}
	Integrating over any interval $(s,t)$ then shows that the pair $(\hat\rho^h,\hat\jmath^h)$ satisfies \eqref{eq:cont-eq}.
\end{proof}



Our main result of this section is the following theorem:

\begin{thm}\label{thm:compactness}
    Let the assumptions of Theorem~\ref{thm:existence-DPA} hold, and let $\{(\hat\rho^h,\hat\jmath^h)\}_{h\in(0,1)}\subset\mathcal{CE}(0,T)$ be the family of density-flux reconstruction associated to a family $\{\x^h\}_{h\in(0,1)}$ of solutions to \eqref{eq:particle_complete} given by Theorem~\ref{thm:existence-DPA}. Then
    there exists a pair $(\rho,j)\in\mathcal{CE}(0,T)$ such that
    \[
   \left. \begin{aligned}
    \hat\rho^h \to \rho\quad &\text{in $L^1([0,T]\times\R)$\,,} \\
    \hat\rho_t^h\rightharpoonup^* \rho_t\quad &\text{weakly-$*$ in $\calM^+(\R)$\quad for every $t\in[0,T]$\,,\; and}\\
    \int_{\cdot}\, \hat\jmath_t^h\,dt \rightharpoonup^* \int_{\cdot}\,j_t\,dt\quad &\text{weakly-$*$ in $\calM([0,T]\times \R)$}
    \end{aligned}\quad \right\}\quad\text{as $h\to 0$\,.}
\]
    for a (not relabelled) subsequence of $\{(\hat\rho^h,\hat\jmath^h)\}_{h\in(0,1)}$.
\end{thm}

For the proof of Theorem~\ref{thm:compactness}, we will be making use of the following generalization of the Aubin--Lions lemma that can be deduced from \cite{RoSa2003}.

\begin{prop}[Generalized Aubin--Lions]\label{thm:aubin}
Let $\{\eta^N\}_{N\in\N}$ be a sequence in $L^{\infty}([0,\,T]; L^1(\R))$ such that
$\eta_t^N \geq 0$ and $\|\eta_t^N\|_{L^1(\R)}=1$ for every $N\in\N$ and $t\in [0,\,T]$.
If the following conditions hold
\begin{enumerate}[label=(\roman*)]
\item $\sup_{N}  \int_0^{T} \left[ \|\eta_t^N\|_{BV(\R)}+ \mathrm{meas}(\mathrm{supp}(\eta_t^N))\right]dt < \infty$\,,
\item there exists a constant $C>0$ independent from $N$ such that
\[
    W_1\big( \eta_t^N, \eta_s^N \big) \le C |t-s|\qquad\text{for all $s,\,t \in (0,\,T)$\,,}
\]
\end{enumerate}
then there exists some $\eta\in L^1([0,T]\times\R)\cap C([0,T];\calM^+(\R))$ and a subsequence (not relabelled) of $\{\eta^N\}_{N\in\N}$ such that 
\[
\left.\begin{aligned}
    \eta^N\to \eta\quad &\text{in $L^1([0,T]\times\R)$\,,\; and}\\
    \eta_t^N\rightharpoonup^* \eta_t\quad &\text{weakly-$*$ in $\calM^+(\R)$\quad for every $t\in[0,T]$}
\end{aligned}\quad\right\}\quad\text{as $N\to\infty$\,.}
\]
\end{prop}

\begin{proof}[Proof of Theorem~\ref{thm:compactness}]
    It is not difficult to see that $\eta^h:=m^{-1}\hat\rho^h$ satisfies $\|\eta_t^h\|_{L^1(\R)}=1$ and
    \begin{align*}
        \sup_{h\in(0,1)}\sup_{t\in[0,T]}\mathrm{meas}(\mathrm{supp}(\eta_t^h)) = \sup_{h\in(0,1)}\sup_{t\in[0,T]}|\x_N^h(t) - \x_0^h(t)|<\infty\,
    \end{align*}
    where we used property {\em(ii)} of Theorem~\ref{thm:existence-DPA}. Therefore, the remaining property to show in condition {\it (i)} of Proposition~\ref{thm:aubin} is 
    \[
        \sup_{h\in(0,1)}\int_0^T \|\eta_t^h\|_{BV(\R)}\,dt<\infty\,,
    \] 
    which we establish in Lemma~\ref{lem:totalvariation} below. Condition {\it (ii)} of Proposition~\ref{thm:aubin} will be established subsequently in Lemma~\ref{thm:weak-compactness} below. An application of Proposition~\ref{thm:aubin} then yields the existence of some $\eta\in L^1([0,T]\times \R)$ and a subsequence (not relabelled) of $\{\eta^h\}_{h\in(0,1)}$, such that
    \[\left.\begin{aligned}
        \hat\rho^h \to \rho := m\eta\quad &\text{in $L^1([0,T]\times\R)$\,,\; and} \\
        \hat\rho_t^h \rightharpoonup \rho_t\quad &\text{weakly-$*$ in $\calM^+(\R)$\quad for every $t\in[0,T]$}
    \end{aligned}\quad\right\}\quad\text{as $h\to 0$}\,.
    \]
    Furthermore, the limiting curve $t\mapsto\rho_t$ is weakly-$*$ continuous.
    
    For the flux, we consider the family $\{J^h\}_{h\in(0,1)}$, where
    \[
        J^h := \int_{\cdot}\, \hat\jmath_t^h\,dt \in \calM([0,T]\times\R)\,.
    \]
    From Lemma~\ref{lem:reconstruction-continuity}, we have that $\sup_{h\in(0,1)} |J^h|([0,T]\times\R)<\infty$, which yields the existence of some $J\in \calM([0,T]\times\R)$ and a subsequence (not relabelled) of $\{J^h\}_{h\in(0,1)}$ such that
    \[
        J^h\rightharpoonup^* J\quad\text{weakly-$*$ in $\calM([0,T]\times\R)$}\,.
    \]
    Since Lemma~\ref{lem:reconstruction-continuity} also provides the uniform integrability of the family $\{t\mapsto|\hat\jmath_t^h|(\R)\}_{h\in(0,1)}$, there exists a Borel measurable family $\{j_t\}_{t\in(0,T)}\in\calM(\R)$ such that $J = \int_{\cdot}\, j_t\,dt$.
    
    To verify that $(\rho,j)\in\mathcal{CE}(0,T)$, it suffices to notice that $(\hat\rho^h,\hat\jmath^h)$ satisfies \eqref{eq:cont-eq} for every $h\in(0,1)$. Passing to the limit with the convergences above yields the assertion, thereby concluding the proof.
\end{proof}

With the following lemmas, we fill in the gap in the proof of Theorem~\ref{thm:compactness}.

\begin{lemma}\label{lem:totalvariation}
Let the assumptions of Theorem~\ref{thm:compactness} be satisfied. Then for every $h\in(0,1)$ we have
\begin{equation}\label{TV}
\|\hat{\rho}_t^h\|_{BV(\R)} \leq c_{BV}\,\|\bar\rho^h\|_{BV(\R)} \qquad \text{for all $t \in [0,T]$}\,
\end{equation}
for some constant $c_{BV}=c_{BV}(\beta,V,W,T)>0$ independent of $h\in(0,1)$.
\end{lemma}
\begin{proof}
We begin by showing that the map $t\mapsto |D\hat\rho_t^h|(\R)$ is absolutely continuous for each $h\in(0,1)$, where the total variation of $D\hat\rho_t^h$ can be explicitly computed, for each $t\in[0,T]$, as
\[
    |D\hat\rho_t^h|(\R) = \rho_0^h(t) + \sum_{i=1}^{N-1} |\rho_i^h(t)-\rho_{i-1}^h(t)| + \rho_{N-1}^h(t)\,.
\]
Indeed, for each $s,t\in(0,T)$, one obtains by the triangle inequality,
\begin{align*}
    \bigl||D\hat\rho_t^h|(\R) - |D\hat\rho_s^h|(\R)\bigr|
    &\le |\rho_0^h(t) - \rho_0^h(s)| + |\rho_{N-1}^h(t) - \rho_{N-1}^h(s)| \\
    &\hspace{4em}+ \sum_{i=1}^{N-1} \Bigl[|\rho_i^h(t)-\rho_i^h(s)|+ |\rho_{i-1}^h(t)-\rho_{i-1}^h(s)|\Bigr]\,.
\end{align*}
We then make use of the following estimate
\begin{align*}
    |\rho_i^h(t)-\rho_i^h(s)| &= \left|\int_s^t \dot\rho_i^h(r)\,dr\right| \le \int_s^t \frac{\rho_i^h(r)}{|K_i|(r)}\left|\dot\x_{i+1}-\dot\x_i\right| dr \\
    &\le 2\frac{M^2}{h}\left(\max\nolimits_{i=0,\ldots,N}|\dot\x_i|\right)|t-s| \qquad\text{for each $i=0,\ldots,N-1$}\,,
\end{align*}
and the bound $|\dot\x_i|\le \beta_{max}|\f_i|\le C$ for some constant $C>0$, independent of $h$ (cf.\ Lemma~\ref{lem:extend-DPA}), to deduce the Lipschitz continuity of $t\mapsto \rho_i^h(t)$, and consequently of $t\mapsto |D\hat\rho_t^h|(\R)$.

Now let us introduce the following time-dependent functions
\[
    s_i := \sign(\rho_{i}^h-\rho_{i-1}^h) - \sign(\rho_{i+1}^h-\rho_{i}^h) = \begin{cases}
        2 & \text{if $\rho_i^h\ge \rho_{i-1}^h$ and $\rho_{i+1}^h\le \rho_i^h$}\,, \\
        -2 & \text{if $\rho_i^h\le \rho_{i-1}^h$ and $\rho_{i+1}^h\ge \rho_i^h$}\,, \\
        0 & \text{otherwise}\,,
    \end{cases}
\]
for $i=1,...,N-2$. At each point of differentiability $t\in(0,T)$, we obtain 
\begin{align}\label{eq:TV-derivative}
     \frac{d}{dt}|D\hat\rho_t^h|(\R)& = \dot\rho_0^h(t) + \sum_{i=1}^{N-1} \sign(\rho_{i}^h(t)-\rho_{i-1}^h(t))(\dot\rho_{i}^h(t) - \dot\rho_{i-1}^h(t)) + \dot\rho_{N-1}^h(t) \notag \\
    &= s_{N-1}\dot\rho_{N-1}^h(t) + s_0\dot\rho_0^h(t) + \sum_{i=1}^{N-2} s_i\dot\rho_{i}^h(t)\,,
\end{align}
where the second equality follows from simply rearranging the terms. 

For each $i=0,\ldots,N-1$, we have that
\[
    \dot\rho_{i}^h = -\frac{\rho_i^h}{|K_i|(t)}\left(\dot\x_{i+1}-\dot\x_i\right) = \frac{\rho_i^h}{|K_i|(t)} A_i - \frac{1}{|K_i|(t)} B_i\,,
\]
where
\[
    A_i := \begin{cases}
        \bigl(\beta(\rho_{1}^h)-\beta(\rho_{0}^h)\bigr)\f_{1}^- + \bigl(\beta(\rho_{0}^h)-\beta_{max}\bigr)\f_{1}^+ & \text{for $i=0$} \\
        \bigl(\beta(\rho_{i+1}^h)-\beta(\rho_i^h)\bigr)\f_{i+1}^-+\bigl(\beta(\rho_i^h)-\beta(\rho_{i-1}^h)\bigr)\f_i^+ &\text{for $i=1,\ldots,N-2$}\,, \\
        \bigl(\beta_{max}-\beta(\rho_{N-1}^h)\bigr)\f_{N-1}^- + \bigl(\beta(\rho_{N-1}^h)-\beta(\rho_{N-2}^h)\bigr)\f_{N-1}^+ & \text{for $i=N-1$}\,,
    \end{cases}
\]
and 
\[
    B_i:= \vartheta(\rho_i^h)(\f_{i+1}-\f_{i})\qquad \mbox{for $ i=0,\ldots,N-1$\,.}
\]
Recall that $\vartheta(\rho)=\rho\beta(\rho)$. For the first term in \eqref{eq:TV-derivative}, we find
\begin{align*}
    s_{N-1}\dot\rho_{N-1}^h &=  s_{N-1}\left(\frac{\rho_{N-1}^h}{\x_{N}-\x_{N-1}}A_{N-1} -\frac{1}{\x_{N}-\x_{N-1}} B_{N-1}\right) \\
    &\le \frac{1}{\x_{N}-\x_{N-1}} |B_{N-1}| \le c_{\f}\beta_{max}\rho_{N-1}^h\,,
\end{align*}
where we used the fact that $s_{N-1} A_{N-1}\le 0$ in the first inequality due to the monotonicity of $\beta$, and Lemma~\ref{lem:liploc_DPA} in the second inequality. The second term in \eqref{eq:TV-derivative} can be estimated similarly to obtain
\[
    s_0\dot\rho_0^h \le c_{\f}\beta_{max}\rho_{0}^h\,.
\]
As for the last term in \eqref{eq:TV-derivative}, we observe that
\[
    s_iA_i \le 0\qquad \text{for $i=1,\ldots,N-2$}\,,
\]
again due to the monotonicity of $\beta$ and the definition of $s_i$,
To estimate the last term, we first rewrite
\begin{equation}\label{splitBi}
\begin{split}
  -\sum_{i=1}^{N-2}s_i\frac{\rho_i^h}{h}B_i &= \underbrace{B.T.+\frac{1}{h}\sum_{i=1}^{N-1}\sign(\rho_{i-1}^h-\rho_i^h)(\rho_i^h-\rho_{i-1}^h)B_i}_{\text{(I)}}\\
  &\qquad +\underbrace{\frac{1}{h}\sum_{i=2}^{N-2}\sign(\rho_{i-1}^h-\rho_i^h)\rho_i^h(B_i-B_{i-1})}_{\text{(II)}}.
\end{split}
\end{equation}
The term (I) can be easily bounded by
\begin{align*}
  |\text{(I)}| &\leq C_1+\frac{1}{h}\sum_{i=1}^{N-1}|\rho_i^h-\rho_{i-1}^h||B_i| \leq C_1+\frac{1}{h}\sum_{i=1}^{N-1}|\rho_i^h-\rho_{i-1}|\rho_i^h\beta(\rho_i^h)|\f_{i+1}-\f_i|\\
 & \leq C_1+c_{\f}\beta_{max}\sum_{i=1}^{N-1}|\rho_i^h-\rho_{i-1}^h|\,, 
\end{align*}
for some constant $C_1>0$, independent of $h\in(0,1)$. Concerning term (II), we write
\begin{align*}
    B_i-B_{i-1}&= \vartheta(\rho_{i-1}^h)(\f_i-\f_{i-1})-\vartheta(\rho_{i}^h)(\f_{i+1}-\f_i)\\
    &=\left(\vartheta(\rho_{i-1}^h)-\vartheta(\rho_i^h)\right)\left(\f_{i+1}-\f_i\right)+\vartheta(\rho_{i-1}^h)\left((\f_i-\f_{i-1})-(\f_{i+1}-\f_i)\right),
\end{align*}
which can be bounded from above, due to Lemma~\ref{lem:liploc_DPA}, by
\begin{align*}
    &\left|\vartheta(\rho_{i-1}^h)-\vartheta(\rho_i^h)\right|\left|\f_{i+1}-\f_i\right| + \vartheta(\rho_{i-1}^h)\left|(\f_i-\f_{i-1})-(\f_{i+1}-\f_i)\right|\\
    &\qquad\leq c_{\f}\Lip(\vartheta)\left|\rho_i^h-\rho_{i-1}^h\right||K_i| + c_{\f}\beta_{max}\rho_{i-1}^h\left( |K_i|^2 + |K_{i-1}|^2 + \bigl||K_i|-|K_{i-1}|\bigr|\right).
\end{align*}
We thus obtain
\[
    |\text{(II)}| \le c_{\f}\bigl(\beta_{max}+\Lip(\vartheta)\bigr)\sum_{i=2}^{N-2}\bigl|\rho_i^h-\rho_{i-1}^h\bigr| + 2c_{\f}\beta_{max}M|\x_N-\x_0|\,.
\]
We can then conclude that 
\begin{align*}
     \left|-\sum_{i=1}^{N-2}s_i\frac{\rho_i^h}{h}B_i\right|& \leq \tilde{c}_1+\tilde{c}_2\sum_{i=2}^{N-2}\left|\rho_i^h-\rho_{i-i}^h\right| \leq \tilde{c}_1+\tilde{c}_2|D\hat{\rho}^h|(\R)\,,
\end{align*}
for some constants $\tilde{c}_1,\tilde{c}_2>0$ independent of $h\in(0,1)$. A simple application of the Gronwall inequality then yields
\[
    |D\hat\rho_t^h|(\R) \le e^{\tilde c_2 t}\bigl( |D\bar\rho^h|(\R) + \tilde c_1 t\bigr) \,,
\]
from which we obtain \eqref{TV} by noting that $\|\hat\rho_t^h\|_{L^1(\R)} = \|\bar\rho^h\|_{L^1(\R)}$ for all $h\in(0,1)$.
\end{proof}


\begin{lemma}\label{thm:weak-compactness}
Let the assumptions of Theorem~\ref{thm:compactness} be satisfied. Then there exists a constant $C>0$, independent of $h\in(0,1)$, such that
    \[
        W_1\left(\frac{1}{m}\hat\rho_t^h,\frac{1}{m}\hat\rho_s^h\right) \le C|t-s|\qquad\text{for all $s,\,t \in (0,\,T)$\,.}
    \]
\end{lemma}
\begin{proof}
    From the continuity equation we find
    \begin{align*}
	|\langle \varphi,\hat\rho_t^h\rangle - \langle \varphi,\hat\rho_t^h\rangle| &= \left|\int_s^t \langle\partial_x\varphi,\hat\jmath_r^h\rangle\,dr\right| \le \|\partial_x\varphi\|_{L^\infty}\int_s^t |\hat\jmath_r^h|(\R)\,dr\,, 
\end{align*}
which, due to Lemma~\ref{lem:reconstruction-continuity}, gives
\[
    |\langle \varphi,\hat\rho_t^h/m\rangle - \langle \varphi,\hat\rho_s^h/m\rangle| \le C\|\partial_x\varphi\|_{L^\infty}|t-s|\,,
\]
for an appopriate constant $C>0$ independent of $h\in(0,1)$ and $t\in[0,T]$. Taking the supremum over Lipschitz functions $\varphi$ satisying $\|\partial_x\varphi\|_{L^\infty}\le 1$ then yields the assertion.
\end{proof}


\section{Convergence to gradient flow and entropic solutions}\label{sec:convergence}

Throughout this section, we assume $\bar\rho$, $\beta$, $V$ and $W$ to satisfy assumptions (In1), (A-$\beta$), (A-V) and (A-Wm) respectively. Furthermore, we consider a family $(\x^h)_{h\in(0,1)}$ of solutions to \eqref{eq:particle_complete} provided by Theorem~\ref{thm:existence-DPA}, and $(\hat\rho^h,\hat\jmath^h)_{h\in(0,1)}$ is the associated density-flux reconstruction, for which we assume to converge to some pair $(\rho,j)\in\mathcal{CE}(0,T)$ in the following sense:
\begin{align}\label{eq:convergences}
   \left. \begin{aligned}
    \hat\rho^h \to \rho\quad &\text{in $L^1([0,T]\times\R)$ and almost surely in $(0,T)\times\R$\,,} \\
    \hat\rho_t^h\rightharpoonup^* \rho_t\quad &\text{weakly-$*$ in $\calM^+(\R)$\quad for every $t\in[0,T]$\,,\; and}\\
    \int_{\cdot}\, \hat\jmath_t^h\,dt \rightharpoonup^* \int_{\cdot}\,j_t\,dt\quad &\text{weakly-$*$ in $\calM([0,T]\times \R)$}\,.
    \end{aligned}\quad \right\}\quad\text{as $h\to 0$\,.}
\end{align}
Note that such a family exists due to the compactness result given in Theorem~\ref{thm:compactness}. Furthermore, due to Theorem~\ref{thm:existence-DPA} we find some bounded domain $\Omega\subset\R$ such that
\[
    \text{$\bigcup_{h\in(0,1)}\text{supp}(\hat\rho_t^h)\cup\text{supp}(\rho_t)\subset\Omega$\quad for all $t\in[0,T]$\,.}
\]


\subsection{Gradient flow solutions}

Before embarking on the proof of Theorems~\ref{mainthm:regular-newtonian} and \ref{mainthm:mild}, let us outline the ingredients involved in the proof.

\medskip

We consider the continuous driving energy \eqref{eq:energy_fun_intro} introduced in Section~\ref{sec:defi}, which we recall here
\begin{align}\label{eq:energy_fun}
    \calF(\rho) = \int_\R V(x)\rho(x)\,dx + \frac{1}{2}\iint_{\R\times\R} W(x-y) \rho(x)\rho(y)\,dx\, dy, 
\end{align}
and the corresponding force
\begin{align}\label{eq:force}
    \sfF_\rho(t,x) = V'(x) + \int_\R W'(x-y)\,\rho_t(y)\,dy\,.
\end{align}

In order to handle the general situation when $W$ only satisfies (A-Wm), we will need to consider an auxiliary driving energy
\begin{align}\label{eq:aux-energy}
    \widehat\calF_h(\rho) = \begin{cases}\displaystyle
        \int_\R V(x)\rho(x)\,dx + \frac{1}{2}\sum_{i=0}^{N-1}\sum_{j\ne i} h^2 \intbar_{K_i}\intbar_{K_j} W(x-y)\,dx\,dy & \displaystyle \text{for $\rho = \sum_{i=0}^{N-1} \frac{h}{|K_i|} \bbmI_{K_i}$\,,} \\
        +\infty & \text{otherwise.}
    \end{cases}
\end{align}
By taking the temporal derivative of the auxiliary driving energy \eqref{eq:aux-energy} along the curve $t\mapsto \hat\rho^h$, and using the continuity equation \eqref{eq:cont-eq}, we arrive at the identity
\begin{align}\label{eq:auxiliary-derivative}
    \frac{d}{dt} \widehat\calF_h(\hat\rho^h(t)) &=  \int_\R \sfF^{h}_{\hat\rho^h}(t,x)\, \hat\jmath_t^h(dx)\,,
\end{align}
where the corresponding auxiliary force is given by
\[
    \sfF^{h}_{\rho}(t,x) := \sum_{i=0}^{N-1}\left(V'(x) + \int_{\R\setminus K_i(t)} W'(x-y)\,\rho(t,y)\,dy\right)\bbmI_{K_i(t)}(x)\,.
\]
Recalling the continuous dual dissipation potential
\[
    \calR^*(\rho,\xi) = \frac{1}{2}\int_\R |\xi(x)|^2\, \vartheta(\rho(x))\,dx
\]
that was introduced in Section~\ref{sec:method}, we combine the results of Sections~\ref{sec:driving-energy}, \ref{sec:action-functional} and \ref{sec:fisher-information} below to obtain the following inequality:
\begin{align*}
    \int_s^t \Bigl\{ \langle \varphi(r,\cdot),\hat\jmath_r^h\rangle -  \calR^*(\hat\rho_r^h,\varphi(r,\cdot))\Bigr\}\, dr + \int_s^t \calR^*(\hat\rho_r^h, -\sfF^{h}_{\hat\rho^h}(r,\cdot))\,dr + \widehat\calF_h(\hat\rho_t^h) - \widehat\calF_h(\hat\rho_s^h) \le O(h)\,,
\end{align*}
which holds for any interval $(s,t)\subset[0,T]$ and any $\varphi\in C_c^{0,1}([0,T]\times\R)$.

Passing to the limit inferior $h\to 0$, and consequently taking the supremum in $\varphi\in C_c^{0,1}([0,T]\times\R)$ results in the so-called {\em Energy-Dissipation inequality}:
\begin{align}\label{eq:EDI}
    \int_s^t \calR(\rho_r,j_r)\, dr + \int_s^t \calR^*(\rho_r, -\sfF_{\rho}(r,\cdot))\,dr + \calF(\rho_t) - \calF(\rho_s) \le 0\,.
\end{align}
The convergences are also justified in Sections~\ref{sec:driving-energy}, \ref{sec:action-functional} and \ref{sec:fisher-information}. However, when $\widehat\calF_h$ is used as driving energy, then the convergence provided in Section~\ref{sec:fisher-information} only holds under an additional assumption on the initial condition $\bar\rho$, namely (In2). Finally, in Section~\ref{sec:chain-rule}, we prove a chain rule to deduce the reverse inequality, and thereby establishing the Energy-Dissipation balance for any $(s,t)\subset[0,T]$:
\begin{align}\label{eq:EDB}\tag{EDB}
    \int_s^t \calR(\rho_r,j_r)\, dr + \int_s^t \calR^*(\rho_r, -\sfF_{\rho}(r,\cdot))\,dr + \calF(\rho_t) - \calF(\rho_s) = 0\,.
\end{align}

\begin{remark}\label{rem:W-free-energy}
    We point out at this point that when $W$ satisfies either (A-Wr) or (A-Wn), then we may use the energy \eqref{eq:energy_fun} in place of the auxiliary energy \eqref{eq:aux-energy}. See Remark~\ref{rem:fisher-information} below for more details.
\end{remark}

\subsubsection{Driving energy}\label{sec:driving-energy}



\begin{lemma}\label{lem:energy-convergence}
    There exists a constant $c_\calF>0$, independent of $h\in(0,1)$, such that
	\begin{align}\label{eq:energy-bound}
		\sup\nolimits_{t\in[0,T]}|\widehat\calF_h(\hat\rho_t^h) - h\calF_h(\x^h(t))| \le c_\calF h\qquad\text{ for all $h\in(0,1)$}\,,
	\end{align}
	where $\calF_h$ is the discrete energy functional defined in \eqref{discrete_ener_fun}.
	
	Furthermore, for every $t\in[0,T]$,
	\[
	    \calF(\rho_t) = \lim_{h\to 0} \widehat\calF_h(\hat\rho_t^h)\,.
	\]
\end{lemma}
\begin{proof}
	We begin by rewriting $\widehat\calF_h(\hat\rho^h)$ in the form
	\begin{align*}
	\widehat\calF_h(\hat\rho^h) &= \sum_{i=0}^{N-1} h\int_0^1 V\bigl(\x_i + \lambda|K_i| \bigr)\,d\lambda + \frac{1}{2}\sum_{i=0}^{N-1}\sum_{j\ne i} h^2\iint W \bigl(\x_i - \x_j + \bigl(\lambda|K_i|-\tau|K_j|\bigr)\bigr)\,d\tau d\lambda \\
	&= h\calF_h(\x) + \sum_{i=0}^{N-1} h \int_0^1 \Bigl[V(\x_i + \lambda|K_i|) - V(\x_i)\Bigr]\,d\lambda\\
	&\qquad + \frac{1}{2}\sum_{i=0}^{N-1}\sum_{j\ne i} h^2\iint \Bigl[W \bigl(\x_i - \x_j + \bigl(\lambda|K_i|-\tau|K_j|\bigr)\bigr) - W(\x_i-\x_j)\Bigr]\,d\tau d\lambda\,.
\end{align*}
Owing to the linear growth assumption on $V$ and $W$, we obtain
\begin{align*}
    |V(\x_i + \lambda|K_i|) - V(\x_i)| &\le |K_i|\int_0^1 \lambda V'(\x_i+\tau\lambda|K_i|)\,d\tau \le c_V |K_i| \bigl(1 + |\x_i| + |\Omega|\bigr)\,,
\end{align*}
and 
\begin{align*}
    &\Bigl|W \bigl(\x_i - \x_j + \bigl(\lambda|K_i|-\tau|K_j|\bigr)\bigr) - W(\x_i-\x_j)\Bigr| \\
    &\hspace{6em} \le \bigl(\lambda|K_i|-\tau|K_j|\bigr)\int_0^1 W'\bigl(\x_i-\x_j + \sigma\bigl(\lambda|K_i|-\tau|K_j|\bigr)\bigr)\bigr)\,d\sigma \\
    &\hspace{6em} \le c_W\bigl(|K_i| + |K_j|\bigr)\bigl(1 + 3 \,|\Omega| \bigr)\,.
\end{align*}
Therefore, we have that
\[
    |\widehat\calF_h(\hat\rho^h) - h\calF_h(\x)| \le c_1 h\,,
\]
for some constant $c_1>0$, independent of $h\in(0,1)$. Since these estimates are independent of $t\in[0,T]$, together, they yield the asserted estimate.

As for the convergence, we begin by noticing that
\[
    |\calF(\hat\rho^h) - \widehat\calF_h(\hat\rho^h)| \le \frac{1}{2}\sum_{i=0}^{N-1} h^2 \intbar_{K_i}\intbar_{K_i} |W(x-y)|\,dx\,dy \le c_2 h\,,
\]
for some constant $c_2>0$, independent of $h\in(0,1)$. On the other hand,
\[
    \calF(\hat\rho^h) = \int_\Omega V(x)\hat\rho^h(x)\,dx + \frac{1}{2}\iint_{\Omega\times\Omega} W(x-y) \hat\rho^h(x)\hat\rho^h(y)\,dx\, dy\,,\qquad\text{for all $h\in(0,1)$}\,.
\]
Therefore, we may take a cut-off function $\chi\in C_c(\R)$ with $\chi\equiv 1$ on $\Omega$, and write
\[
     \calF(\hat\rho^h) = \int_\R \chi(x)V(x)\hat\rho^h(x)\,dx + \frac{1}{2}\iint_{\R\times\R} \chi(x)\chi(y)W(x-y) \hat\rho^h(x)\hat\rho^h(y)\,dx\, dy\,.
\]
Since $\hat\rho_t^h\rightharpoonup^*\rho_t$ weakly-$*$ in $\calM^+(\R)$ as $h\to0$ for every $t\in[0,T]$, and the integrands are continuous functions with compact support, we may pass $h\to0$ to obtain the desired limit.
\end{proof}

\subsubsection{Action functional}\label{sec:action-functional}

\begin{lemma}\label{lem:action}
	For any $(s,t)\subset[0,T]$ and $\varphi\in C_c^{0,1}([0,T]\times\R)$:
	\begin{align}\label{eq:action-bound}
		\int_s^t \langle \varphi(r,\cdot),\hat\jmath_r^h\rangle - \calR^*(\hat\rho_r^h,\varphi(r,\cdot))\, dr = h\int_s^t\calR_h(\x^h(r),\j^h(r))\,dr + O(h)|_{h\to 0}\,,
	\end{align}
	where $\calR_h$ is the discrete dissipation potential given in \eqref{eq:discrete-dissipation-potential}.
	
	Moreover, the following limit holds:
	\[
	    \int_s^t \langle \varphi(t,\cdot),j_t\rangle -  \calR^*(\rho_r,\varphi(r,\cdot))\, dr = \lim_{h\to 0} \int_s^t \langle \varphi(r,\cdot),\hat\jmath_r^h\rangle - \calR^*(\rho_r,\varphi(r,\cdot))\, dr\,.
	\]
\end{lemma}
\begin{proof}
The limit clearly holds due to the convergences assumed for $(\hat\rho^h,\hat\jmath^h)$. As for the estimate, we simply make use of Taylor's formula to deduce
\begin{align*}
    \int_\R |\varphi(x)|^2 \vartheta(\hat\rho^h(x))\,dx &= h\sum_{i=0}^{N-1} \beta(\rho_i^h)\left\{\int_0^1 |(\varphi(\x_i^\lambda))^+|^2\,d\lambda 
    + \int_0^1 |(\varphi(\x_i^\lambda))^-|^2\,d\lambda\right\} \\
    &= h\sum_{i=0}^{N-1} \beta(\rho_i^h) \Bigl\{|(\varphi(\x_i))^+|^2 + |(\varphi(\x_{i+1}))^-|^2\Bigr\} + O(h) \\
    &= h\sum_{i=0}^{N-1} \Bigl\{ \beta(\rho_i^h) |(\varphi(\x_i))^+|^2 + \beta(\rho_{i-1}^N) |(\varphi(\x_i))^-|^2\Bigr\} + O(h) \\
    &= h\calR_h^*(\x^h,\xi) + O(h)\,,\qquad \xi_i = \varphi(\x_i)\,,
\end{align*}
where $\x_i^\lambda = (1-\lambda)\x_i + \lambda \x_{i+1}$, and we used the estimates
\begin{align*}
	(\varphi(\x_i^\lambda))^+ &= (\varphi(\x_i + \lambda|K_i|)^+ = (\varphi(\x_i))^+ + O(|K_i|)\,, \\
	(\varphi(\x_i^\lambda))^- &= (\varphi(\x_{i+1} - (1-\lambda)|K_i|)^- = (\varphi(\x_{i+1}))^- + O(|K_i|)\,.
\end{align*}
In a similar fashion we can estimate
\[
	\langle \varphi,\hat\jmath^h\rangle = h\sum_{i=0}^{N-1}  \varphi(\x_i)\,\j_i + O(h)\,.
\]
Since the estimates hold for almost every $t\in(0,T)$, we may put the terms together to obtain
\begin{align*}
	\langle \varphi(t,\cdot),\hat\jmath_t^h\rangle -  \frac{1}{2}\int_\R |\varphi(t,x)|^2 \vartheta(\hat\rho_t^h(x))\,dx &= h \bigl(\langle\xi,\j^h(t)\rangle - \calR_h^*(\x^h(t),\xi)\bigr) + O(h) \\
	&\le h\calR_h(\x,\j) + O(h)\,,
\end{align*}
where the inequality follows from the definition of the Legendre dual. Integrating over $(0,T)$ then yields the required estimate.
\end{proof}

\subsubsection{Fisher information}\label{sec:fisher-information}

In order to recover the required limit for the Fisher information when $W$ only satisfies (A-Wm), we will have to assume additionally that $\bar\rho$ satisfies (In2). Under this assumption, we have uniform control over the size $|K_i|$ of the intervals $K_i$. Indeed, Theorem~\ref{thm:upper-bound} gives the bound
\begin{align}\label{eq:upper-bound-2}
    \sup_{t\in[0,T]}\sup_{i=0,\ldots,N-1}|K_i|(t) \le e^{\mu T}\frac{h}{\sigma}\,,
\end{align}
for some constant $\mu>0$, independent of $h\in(0,1)$. In particular, we have the following result.

\medskip

We note that this is the only place throughout this section where (In2) is assumed. We believe, however, that this assumption may be removed, but would require a more careful construction of the initial datum for the DPA in the case when $\bar\rho$ has non-connected supports. 

\begin{lemma}\label{lem:fisher-information}
    There exists a constant $c_{\mathcal{D}}>0$, independent of $h>0$, such that
    \begin{align}\label{eq:fisher-information-bound}
       \int_0^T \calR^*(\hat\rho_t^h,-\sfF^{h}_{\hat\rho^h}(t,\cdot))\,dt 
       \le h\int_0^T\mathcal{R}_h^*(\x^h(t),-\sfF^h(\x(t)))\,dt + c_{\mathcal{D}} h\,,
    \end{align}
    where $\calR_h^*$ is the discrete dual dissipation given in \eqref{eq:discrete-dual-dissipation}.
    
    If $\bar\rho$ satisfies additionally (In2), then the following convergence holds for any $(s,t)\subset[0,T]$:
    \[
        \int_s^t \calR^*(\rho_t,-\sfF_{\rho}(t,\cdot))\,dt \le \lim_{h\to 0}\int_s^t \calR^*(\hat\rho_t^h,-\sfF^{h}_{\hat\rho^h}(t,\cdot))\,dt
    \]
\end{lemma}
\begin{proof}
From the properties $V''\in L^\infty(\R)$ and $W''\in L^\infty(\R\setminus\{0\})$, we find for any $x\in K_i$,
\[
    \bigl|\sfF^{h}_{\hat\rho^h}(t,x) - \f_i\bigr|\le c_1\bigl(|K_i| + h\bigl),\qquad \bigl|\sfF^{h}_{\hat\rho^h}(t,x) - \f_{i+1}\bigr| \le c_1\bigl(|K_i| + h\bigl)\,,
\]
for some constant $c_1>0$, independent of $h\in(0,1)$. 
In particular, we have
\begin{align*}
    \bigl|( \sfF^{h}_{\hat\rho^h}(t,x))^-\bigr|^2 
    &\le \bigl|( \sfF^{h}_{\hat\rho^h}(t,x))^- - \f_i^- \bigr|^2 + 2\bigl|( \sfF^{h}_{\hat\rho^h}(t,x))^- - \f_i^- \bigr||\f_i^-| + |\f_i^-|^2 \\
    &\le c_2\bigl(|K_i| + h\bigr) + |\f_i^-|^2\,,
\end{align*}
for some constant $c_2>0$, independent of $h\in(0,1)$. Similarly, we deduce
\[
    \bigl|( \sfF^{h}_{\hat\rho^h}(t,x))^+\bigr|^2 \le c_2\bigl(|K_i| + h\bigr) + |\f_{i+1}^+|^2\,,
\]
with the same constant $c_2$ as above. Consequently, we find that
\begin{align*}
    \int_\R \Bigl| \sfF^{h}_{\hat\rho^h}(t,x)\Bigr|^2 \vartheta(\hat\rho_t^h(x))\,dx &= \int_\R \Bigl|( \sfF^{h}_{\hat\rho^h}(t,x))^-\Bigr|^2 \vartheta(\hat\rho_t^h(x))\,dx + \int_\R \Bigl|( \sfF^{h}_{\hat\rho^h}(t,x))^+\Bigr|^2 \vartheta(\hat\rho_t^h(x))\,dx \\
    &\le \sum_{i=0}^{N-1} h\beta(\rho_i^h(t))|\f_i^-|^2 + \sum_{i=0}^{N-1} h\beta(\rho_i^h(t))|\f_{i+1}^+|^2 + 2c_2\beta_{max}\bigl(m+ |\Omega|\bigr)h\\
    &= h\mathcal{R}_h^*(\x^h(t),-\sfF^h(\x(t))) + c_3h\,,
\end{align*}
for some constant $c_3>0$, independent of $h\in(0,1)$, which establishes \eqref{eq:fisher-information-bound} after integrating in $t$.

As for the convergence, let $t\in(0,T)$ and $x\in \text{interior}(\text{supp} (\rho_t))$ be arbitrary, but fixed. Then for $h_0\ll 1$ sufficiently small (which may depend on $x$), $x\in \text{interior}(\text{supp}(\hat\rho_t^h))$ for all $h\in(0,h_0)$. Further, let $\{K^h\}_{h\in(0,h_0)}$ be a family of sets, where $K^h\in \{K_i^h(t)\}_{i=0,\ldots,N-1}$ and such that $x\in K^h$ for every $h\in(0,h_0)$. We may then write
\[
    \sfF^{h}_{\hat\rho^h}(t,x) = V'(x) + \int_{\R\setminus K^h} W'(x-y)\,\hat\rho_t^h(y)\,dy\qquad\text{for any $h\in(0,h_0)$\,.}
\]
From the upper bound \eqref{eq:upper-bound-2}, we find that $|K^h|\to 0$ for $h\to 0$, and hence, $\bbmI_{\R\setminus K^h} \to \bbmI_\R$ almost surely as $h\to0$. Along with the almost sure convergence of $\hat\rho^h\to\rho$, and the estimate
\[
    \int_{\R\setminus K^h} |W'(x-y)|\,\hat\rho_t^h(y)\,dy \le c_W\bigl(1 + |\Omega|\bigr) \int_\Omega \,\hat\rho_t^h(y)\,dy\,,
\]
we can then conclude that
\[
    \sfF^{h}_{\hat\rho^h}(t,x) \to \sfF_\rho(t,x)\quad\text{for $\rho$-a.e.\ $(t,x)\in (0,T)\times\R$\quad as $h\to0$}\,,
\]
by means of the Dominated Convergence Theorem. Along with the bound
\[
    |\sfF^{h}_{\hat\rho^h}(t,x)| \le c_{\sfF}\qquad\text{for every $(t,x)\in(0,T)\times\R$\,,}
\]
for some constant $c_{\sfF}>0$, independent of $h\in(0,1)$, and the almost sure convergence of $\vartheta(\hat\rho^h)\to\vartheta(\rho)$, we may apply the Dominated Convergence Theorem again to obtain the asserted convergence.
\end{proof}

\begin{remark}\label{rem:fisher-information}
    We note that if $W$ satisfies (A-Wr) or (A-Wn), we can simply consider the driving energy $\calF$ \eqref{eq:energy_fun}, and hence the corresponding force $\sfF_\rho$ \eqref{eq:force}. In this case, estimate \eqref{eq:fisher-information-bound} holds true with $\sfF_\rho$ instead of $\sfF_\rho^h$, and assumption (In2) on $\bar\rho$ is not required to obtain the limit provided in Lemma~\ref{lem:fisher-information}. 
    
    Indeed, when $W$ satisfies (A-Wn) (cf.\ Remark~\ref{rem:Newtonian-Lipschitz}), one has the identity
    \begin{align*}
        \sfF_{\hat\rho^h}(t,x) -\f_i = V'(x) - V'(\x_i)\, \pm\, h\frac{x-\x_{i-1}}{|K_i|}\,,\qquad\text{for any $x\in K_i$}\,,
    \end{align*}
    from which one obtains the necessary bounds.
\end{remark}

\subsubsection{Chain rule}\label{sec:chain-rule}
Here, we prove a chain rule for a large class of pairs $(\rho,j)\in\mathcal{CE}(0,T)$, which establishes the reverse inequality for the Energy-Dissipation inequality \eqref{eq:EDI}.

\begin{lemma}[Chain rule]\label{lem:chain-rule}
    Let $(\mu,\xi)\in\mathcal{CE}(0,T)$ be a pair such that 
    \[
     \mu_t\in L^1(\R)\quad\text{and\quad $\text{supp}(\rho_t)\subset \Omega$\quad for all $t\in[0,T]$}\,,
    \]
    where $\Omega\subset\R$ is a bounded domain, and that
    \[
         \int_0^T \int_{\R} \left\{ \left|\frac{d\xi_r}{d\mathfrak{m}_r} \right|^2 + \bigl|\sfF_\mu(t,x)\bigr|^2\right\}\mathfrak{m}_t(dx)\,dr <\infty\,,\qquad \mathfrak{m}_t(dx) = \vartheta(\mu_t(x))\,dx\,.
    \]
    Then the map $(0,T)\ni t\mapsto \calF(\mu_t)\in\R$ is absolutely continuous, and for any $(s,t)\subset[0,T]$:
    \[
        \frac{1}{2}\int_s^t \int_{\R} \left\{ \left|\frac{d\xi_r}{d\mathfrak{m}_r} \right|^2 + \bigl|\sfF_\mu(t,x)\bigr|^2\right\}\mathfrak{m}_t(dx)\,dr + \calF(\mu_t)\ge \calF(\mu_s)\,.
    \]
\end{lemma}
\begin{proof}
    From the continuity equation and the property $\int_0^T |\xi_t|(\R)\,dt$, one easily deduces that the curve $t\mapsto\eta_t:=\mu_t/m$ is absolutely continuous with respect to the 1-Wasserstein distance $W_1$. Now consider a regularization of $W^\varepsilon$ such that $W^\varepsilon$ converges to $W$ pointwise almost everywhere, and $W^\varepsilon$ satisfies the linear growth assumption
    \[
        |(W^\varepsilon)'(r)| \le c\bigl(1+|r|\bigr),\qquad |r|>0\,,
    \]
    for some constant $c>0$ independent of $\varepsilon>0$. It is not difficult to see that the functional
    \[
        \calF^\varepsilon(\rho) = \int_\Omega V(x)\,\rho(dx) + \iint_{\Omega\times\Omega} W^\varepsilon(x-y)\,\rho(dx)\,\rho(dy)
    \]
    is Lipschitz with respect to the 1-Wasserstein metric, and the identity
    \[
        \calF^\varepsilon(\mu_t) - \calF^\varepsilon(\mu_s) = \int_s^t\int_\Omega \left\{V'(x) + \int_{\Omega} (W^\varepsilon)'(x-y)\,\mu_r(dy)\right\} \xi_r(dx)\,dr\,,
    \]
    holds for any $(s,t)\subset[0,T]$. Invoking the Dominated Convergence Theorem, we may then pass to the limit to obtain
    \begin{align*}
        \calF(\mu_t) - \calF(\mu_s) &= \int_s^t\int_\Omega \sfF_\mu(r,x)\, \xi_r(dx)\,dr = - \int_s^t\int_\Omega \bigl(-\sfF_\mu(r,x)\bigr)\, \frac{d\xi_r}{d\mathfrak{m}_r}(x)\,\mathfrak{m}_r(dx)\,dr \\
        &\ge -\frac{1}{2} \int_s^t\int_{\R} \left\{ \left|\frac{d\xi_r}{d\mathfrak{m}_r} \right|^2 + \bigl|\sfF_\mu(r,x)\bigr|^2\right\}\mathfrak{m}_r(dx)\,dr\,,
    \end{align*}
  where the inequality follows from Young's inequality.
\end{proof}

\subsubsection{Proof of Theorem~\ref{mainthm:mild}}\label{sec:proof-B}
Let the assumptions of Theorem~\ref{mainthm:mild} hold, and consider a sequence of pairs $\{(\hat\rho^h,\hat\jmath^h)\}_{h\in(0,1)}$ satisfying the convergences given in \eqref{eq:convergences} towards some $(\rho,j)\in\mathcal{CE}(0,T)$. In this case, Lemmas \ref{lem:energy-convergence}, \ref{lem:action} and \ref{lem:fisher-information} apply. Putting the estimates given in those lemmas together, we arrive at the following inequality for any $(s,t)\subset[0,T]$ and $\varphi\in C_c^{0,1}([0,T]\times\R)$:
\begin{align*}
    &\int_s^t \langle \varphi(r,\cdot),\hat\jmath_r^h\rangle - \calR^*(\hat\rho_r^h,\varphi(r,\cdot))\, dr + \int_0^T \calR^*(\hat\rho_t^h,-\sfF^{h}_{\hat\rho^h}(r,\cdot))\,dr + \widehat\calF_h(\hat\rho_t^h) - \widehat\calF_h(\hat\rho_s^h) \\
    &\quad\le h\left\{\int_s^t\calR_h(\x^h(r),\j^h(r))\,dr+ \int_s^t\mathcal{R}_h^*(\x^h(r),-\sfF^h(\x(r)))\,dr + \calF_h(\x^h(t)) - \calF_h(\x^h(s))\right\} + O(h)\,.
\end{align*}
The first term on the right-hand side vanishes due to the discrete Energy-Dissipation balance \eqref{eq:discrete-EDB}. Taking the limit inferior $h\to 0$, and making use of Lemmas \ref{lem:energy-convergence}, \ref{lem:action} and \ref{lem:fisher-information} again, we obtain 
\begin{align}\label{eq:pre-EDI}
    \int_s^t \langle \varphi(r,\cdot),j_r\rangle - \calR^*(\rho_r,\varphi(r,\cdot))\, dr + \int_0^T \calR^*(\rho_t,-\sfF_{\rho}(t,\cdot))\,dt + \calF(\rho_t) - \calF(\rho_s)\le 0\,.
\end{align}
The first term on the left-hand side may be rewritten as
\[
    I(J,\varTheta;\varphi):=\iint_{(s,t)\times\R} \varphi(r,x)\,J(dr\,dx) - \frac{1}{2}  \iint_{(s,t)\times\R} |\varphi(r,x)|^2\, \varTheta(dr\,dx)\,,
\]
where $J=\int_{\cdot}\,j_t\,dt$ and $\varTheta = \int_{\cdot}\,\vartheta(\rho_t)\,dt$. Since the driving energies $\calF(\rho_t)$ are finite for every $t\in[0,T]$, and that $\calR^*(\rho_t,-\sfF_{\rho}(t,\cdot))$ is non-negative, we then find that
\[
    \mathscr{I}(s,t):=\sup\, \Bigl\{ I(J,\varTheta;\varphi)\,:\, \varphi\in C_c^{0,1}((s,t)\times\R)\Bigr\} <+\infty\,.
\]
The first observation from this fact, is that $J\ll \varTheta$, i.e.\ $J$ is absolutely continuous with respect to $\varTheta$. In particular, the Radon-Nikodym derivative $dJ/d\varTheta$ exists. Furthermore, due to density, the supremum may be taken over functions $\varphi\in C_c([0,T]\times\R)$. On the other hand, a simple computation yields
\[
    \iint_{(s,t)\times\R}  \varphi(r,x)\,J(dr\,dx) \le \|\varphi\|_{L^2((s,t)\times\R,\varTheta)}  \sqrt{2\mathscr{I}(s,t)}\qquad\text{for every $\varphi\in C_c((s,t)\times\R)$\,.}
\]
Since $\varTheta$ has a Lebesgue density on $(s,t)\times\R$, $\calC_c([0,T]\times\R)$ is dense in $L^2((s,t)\times\R,\varTheta)$, which consequently implies the estimate
\[
     \left\|\frac{dJ}{d\varTheta}\right\|_{L^2((s,t)\times\R,\varTheta)} \le \sqrt{2\mathscr{I}(s,t)}\quad\Longleftrightarrow\quad \frac{1}{2}\left\|\frac{dJ}{d\varTheta}\right\|_{L^2((s,t)\times\R,\varTheta)}^2 \le \mathscr{I}(s,t)\,.
\]
Altogether, we then obtain the Energy-Dissipation inequality \eqref{eq:EDI}.

Since the pair $(\rho,j)\in\mathcal{CE}(0,T)$ satisfies the assumptions of Lemma~\ref{lem:chain-rule}, we also obtain the chain rule \eqref{eq:intro-chainrule} along with the reverse inequality, thereby giving the Energy-Dissipation balance \eqref{eq:EDB}. The uniform upper bound $\rho(t,x)\le M$ for almost every $(t,x)\in[0,T]\times\R$ follows directly from the fact that the same bound holds uniformly in $h\in(0,1)$ for the sequence $\{\hat\rho^h\}_{h\in(0,1)}$.
\hfill\qedsymbol

\subsection{Entropy solutions} This section is devoted to proving  the convergence of the piecewise constant discrete density $\hat\rho^h$ to the entropy solutions of \eqref{eq:main} in the sense of Definition \ref{def:entropy_sol}. More precisely, we are going to show that the entropy condition is satisfied in the limit, namely, setting
\begin{align*}
\eta_c(s)&:=|s-c|,\qquad \eta'_c(s)=\sign(s-c),
\end{align*} 
we prove that for any constant $c>0$ and every $\varphi\in C_c^\infty([0,T)\times\R)$ with $\varphi\ge 0$,
\begin{align}\label{eq:entropy_disc}
    \begin{aligned}
       0&\leq \liminf_{h\to 0}\left\{ \int_\R \eta_c(\hat\rho_0^h)\,\varphi(0,\cdot)\,dx + \int_0^T\int_\R \eta_c(\hat{\rho}^h)\,\partial_t\varphi\,dx\,dt\right.\\
       &\hspace{6em}-\left.\int_0^T\int_\R\eta'_c(\hat\rho^h)\left[\left(\vartheta(\hat\rho^h)-\vartheta(c)\right)\sfF_{\hat\rho^h}\,\partial_x\varphi-\vartheta(c)\,\partial_x\sfF_{\hat\rho^h}\,\varphi\right] dx \,dt\right\},
    \end{aligned}
\end{align}
where for $W$ under the assumption (A-Wn), the corresponding force $\sfF_\rho$ takes the explicit form \eqref{eq:newtonian-force}


\medskip

To do so, we first provide some preliminary results.

\begin{lemma}\label{A1}
For any constant $c>0$ and for all $\varphi \in C_c^\infty(\left[0,T\right]\times\R)$ with $\varphi\geq 0$, we have that
\begin{equation*}
    \left|\sum_{i=0}^{N-1}A_i^1\right|\to 0\quad \mbox{as}\quad h\to 0\,,
\end{equation*}
where (dropping the $t$ variable)
\begin{equation}\label{eq:A1}
     A_i^1 := \int_0^T \sic\,\rho_i^h\left(\dot{\x}_{i+1}-\dot{\x}_{i}\right)\left(\intbar_{K_i} \varphi(x) \,dx -\varphi(\x_{i+1})\right) dt\,.
\end{equation}
\end{lemma}
\begin{proof}
A direct estimate shows that
\begin{align*}
  & \sum_{i=0}^{N-1}\int_0^T \left|\rho_i^h\left(\dot{\x}_{i+1}-\dot{\x}_{i}\right)\left(\intbar_{K_i} \varphi \,dx -\varphi(\x_{i+1})\right)\right|dt\\
    &\hspace{10em}\leq  \,h \|\partial_x\varphi\|_{\infty}\left(c_{BV}\Lip(\beta)\|\f\|_{\infty}\|\bar\rho\|_{BV(\R)}+ \beta_{max}c_{\f}|\Omega|\right)T,
\end{align*}
where we used the bound
\begin{align*}
     \left|\intbar_{K_i} \varphi \,dx -\varphi(\x_{i+1})\right|
     \leq \frac12 \|\partial_x\varphi\|_{\infty}|K_i|.
\end{align*}
and the bounds provided in Lemmas \ref{lem:extend-DPA} and \ref{lem:totalvariation}.
\end{proof}
\begin{lemma}\label{A2}
For any constant $c>0$ and for all $\varphi \in C_c^\infty([0,T)\times\R)$ with $\varphi\geq 0$, we have that
\begin{equation*}
    \left|\sum_{i=0}^{N-1}A_i^2\right| \to 0 \quad\mbox{as}\quad h\to 0\,,
\end{equation*}
where (dropping the $t$ variable)
\begin{equation}\label{eq:A2}
    A_i^2 := \int_0^T \sic\left(\varphi(\x_i)-\varphi(\x_{i+1})\right)\left[\rho_i^h\,\dot{\x}_i+\vartheta(\rho_i^h)\,\sfF_{\hat\rho^h}(\x_i)\right] dt\,.
\end{equation}
\end{lemma}
\begin{proof}
    By using the regularity of the test function and the fact that $\sfF_{\hat\rho^h}(\x_i)=\f_i+O(h)$, we obtain the following estimate:
\begin{align*}
    & \sum_{i=0}^{N-1}\int_0^T \left|\left(\varphi(\x_i)-\varphi(\x_{i+1})\right)\left[\rho_i^h\dot{\x}_i+\vartheta(\rho_i^h)\,\sfF_{\hat\rho^h}(\x_i)\right]\right| dt\\
        &\hspace{8em}\leq  h\|\partial_x\varphi\|_{\infty}\sum_{i=0}^{N-1}\int_0^T  \left|\dot{\x}_i+\beta(\rho_i^h)\f_i\right|\,dt + O(h)\,.
        \end{align*}
Using the equation for $\dot{\x}_i$, the Lipschitz continuity of $\beta$, and Lemma~\ref{lem:totalvariation} we get
        \begin{align*}
        \sum_{i=0}^{N-1}\int_0^T  \left|\dot{\x}_i+\beta(\rho_i^h)\f_i\right|\,dt &= \sum_{i=0}^{N-1}\int_0^T  \left|-\beta(\rho_i^h)\f_i^--\beta(\rho_{i-1}^h)\f_i^++\beta(\rho_i^h)\left(\f_i^++\f_i^-\right)\right|\,dt \\
         &\leq \sum_{i=0}^{N-1}\int_0^T \left|\beta(\rho_i^h)-\beta(\rho_{i-1}^h)\right||\f_i^+|\,dt \le c_{BV}\|\f\|_{\ell^\infty} \Lip(\beta) \|\bar\rho\|_{BV(\R)}\, T\,,
\end{align*}
which, together with the previous estimate, proves the desired convergence.
\end{proof}

\begin{lemma}\label{A3}
For any constant $c>0$ and for all $\varphi \in C_c^\infty([0,T)\times\R)$ with $\varphi\geq 0$, 
we have that
\begin{align*}
    \liminf_{h\to 0} \sum_{i=0}^{N} A_i^3 &\geq 0\,,
\end{align*}
where (dropping the $t$ variable)
\begin{equation}\label{eq:A3}
     A_i^3 :=\int_0^T\left(\sicm-\sic\right)\varphi(\x_i)\left[\dot{\x}_i+\beta(c)\sfF_{\hat\rho^h}(\x_i)\right] dt\,,
\end{equation}
with the convention that $\rho_{-1}^h = \rho_N^h = 0$.
\end{lemma}
\begin{proof}
A direct computation as in the previous lemma shows that
    \begin{align*}
     &\int_0^T\left(\sicm-\sic\right)\varphi(\x_i)\left[\dot{\x}_i+\beta(c)\,\sfF_{\hat\rho^h}(t,\x_i)\right]\,dt\\
     &\qquad = \int_0^T\left(\sicm-\sic\right)\varphi(\x_i)\left[\dot{\x}_i+\beta(c)\f_i\right]\,dt + O(h)\\
     &\qquad = \int_0^T\left(\sicm-\sic\right)\varphi(\x_i)\left[\left(\beta(c)-\beta(\rho_i^h)\right)\f_i^-+\left(\beta(c)-\beta(\rho_{i-1}^h)\right)\f_i^+\right] dt + O(h)\,.
\end{align*}
Notice that if either $\rho_i^h, \rho_{i-1}^h>c$ or $\rho_i^h, \rho_{i-1}^h<c$, then the integral on the right-hand side is zero. In the other cases, the non-negativity follows from the monotonicity of $\beta$.
\end{proof}

\begin{lemma}\label{A4}
For any constant $c>0$ and for all $\varphi \in C_c^\infty(\left[0,T\right]\times\R)$ with $\varphi\geq 0$, we have that 
\begin{align*}
      \left|\sum_{i=0}^{N-1}A_i^4\right| \to 0 \mbox{ as } h\to 0\,,
\end{align*}
where (dropping the $t$ variable)
\begin{equation}\label{eq:A4}
     A_i^4 :=\int_0^T\sic\,\vartheta(\rho_i^h)\left[\int_{K_i} (\partial_x\sfF_{\hat\rho^h})(x)\,\varphi(x) \,dx  - \varphi(\x_i)\left(\sfF_{\hat\rho^h}(\x_{i+1})-\sfF_{\hat\rho^h}(\x_i)\right) \right] dt\,.
\end{equation}
\end{lemma}
\begin{proof}
Under assumption (A-Wr) the map $x\mapsto\sfF_{\hat\rho^h}(t,x)$ is an element of the Sobolev space $W^{2,\infty}(\R)$. Therefore, we can apply Taylor's theorem to obtain
\begin{align*}
    \intbar_{K_i} \partial_x\sfF_{\hat\rho^h}\,\varphi \,dx  - \varphi(\x_i)\frac{\sfF_{\hat\rho^h}(\x_{i+1})-\sfF_{\hat\rho^h}(\x_i)}{|K_i|} &= \intbar_{K_i} \left\{ \partial_x\sfF_{\hat\rho^h}(x) - \frac{\sfF_{\hat\rho^h}(\x_{i+1})-\sfF_{\hat\rho^h}(\x_i)}{|K_i|}\right\} \varphi(x) \,dx \\
    &\qquad + \frac{\sfF_{\hat\rho^h}(\x_{i+1})-\sfF_{\hat\rho^h}(\x_i)}{|K_i|} \left\{ \intbar_{K_i} \varphi(x)\,dx - \varphi(\x_i)\right\} \\
    &= c_1\|\varphi\|_{W^{1,\infty}(\R)}|K_i|\,,
\end{align*}
for some constant $c_1>0$, independent of $h\in(0,1)$. As for the Newtonian potential, i.e.\ $W$ satisfying (A-Wn), we use the explicit formula in \eqref{eq:newtonian-force} to find
\begin{align*}
    \intbar_{K_i} \partial_x\sfF_{\hat\rho^h}^W\,\varphi \,dx  - \varphi(\x_i)\frac{\sfF_{\hat\rho^h}^W(\x_i)-\sfF_{\hat\rho^h}^W(\x_{i+1})}{|K_i|} = \pm \intbar_{K_i}\bigl[\varphi(x)-\varphi(\x_i)\bigr]\,\hat\rho^h(x)\,dx,\,
\end{align*}
where $\sfF_{\hat\rho^h}^W$ denotes the interaction part of the force $\sfF_{\hat\rho^h}$, and the sign depends on whether $W$ is attractive or repulsive. In either cases, we obtain a similar estimate as in the (A-Wr) case, which gives
   \begin{align*}
     \left|\sum_{i=0}^{N-1}A_i^4\right| &\leq h\beta_{max}\sum_{i=0}^{N-1} \int_0^T \left| \intbar_{K_i} \partial_x\sfF_{\hat\rho^h}\,\varphi \,dx  - \varphi(\x_i)\frac{\sfF_{\hat\rho^h}(\x_{i+1})-\sfF_{\hat\rho^h}(\x_i)}{|K_i|}\right| dt\\
   &\leq c_1 h\beta_{max}|\Omega|T\,,
\end{align*}
thereby concluding the proof, after passing to the limit $h\to 0$.
\end{proof}

\subsubsection{Proof of Theorem~\ref{mainthm:regular-newtonian}}\label{sec:proof-A}

We are now in the position of proving the inequality \eqref{eq:entropy_disc}, and consequently (ii) of Theorem~\ref{mainthm:regular-newtonian}. The proof of Theorem~\ref{mainthm:regular-newtonian}(i) follows from the proof of Theorem~\ref{mainthm:mild} with the driving functional $\calF$ and force $\sfF_\rho$ in place of $\widehat\calF_h$ and $\sfF_\rho^h$ respectively (cf.\ Remarks~\ref{rem:W-free-energy} and \ref{rem:fisher-information}).

Let us now prove inequality \eqref{eq:entropy_disc}. We begin by splitting the integral in \eqref{eq:entropy_disc} as follows 
\begin{align*}
 &\int_\R \eta_c(\hat\rho_0^h)\,\varphi(0,\cdot)\,dx + \int_0^T\int_\R \eta_c(\hat\rho^h)\,\partial_t\varphi\,dx\,dt\\
 &\qquad - \int_0^T\int_\R \eta'_c(\hat\rho^h)\left[\left(\vartheta(\hat\rho^h)-\vartheta(c)\right)\sfF_{\hat\rho^h}\,\partial_x\varphi-\vartheta(c)\,\partial_x \sfF_{\hat\rho^h}\,\varphi\right] dx \,dt
=: B^1 + \sum_{i=0}^{N-1} I_i^1 + I_i^2,
\end{align*}
where we have defined
\begin{align*}
    I_i^1 &:= \int_{K_i} \eta_c(\bar\rho_i^h)\,\varphi(0,x)\,dx + \int_0^T\int_{K_i} \eta_c(\rho_i^h)\,\partial_t\varphi\,dx \,dt\,,\\
     I_i^2 &:= -\int_0^T\int_{K_i}\sic\left[\left(\vartheta(\rho_i^h)-\vartheta(c)\right)\sfF_{\hat\rho^h}\,\partial_x\varphi-\vartheta(c)\,\partial_x \sfF_{\hat\rho^h}\,\varphi\right] dx \,dt\,,\\
    B^1 &:= \int_{\R\setminus[\bar\x_0,\bar\x_N]} c \varphi(0,\cdot)\,dx + \int_0^T\int_{-\infty}^{\x_0} c\partial_t\varphi -  \vartheta(c)\left(\sfF_{\hat\rho^h}\,\partial_x\varphi+\,\partial_x \sfF_{\hat\rho^h}\,\varphi\right) dx\,dt \\
    &\qquad + \int_0^T\int_{\x_N}^{+\infty} c\partial_t\varphi -  \vartheta(c)\left(\sfF_{\hat\rho^h}\,\partial_x\varphi+\partial_x \sfF_{\hat\rho^h}\,\varphi\right) dx\,dt\,.
\end{align*}
On performing integration by parts in $t$ and $x$, and rearranging the terms appropriately, we obtain the  expressions
\begin{align*}
    I_i^1 &= \int_0^T \sic\,\rho_i^h\left(\dot{\x}_{i+1}-\dot{\x}_{i}\right)\left(\intbar_{K_i} \varphi \,dx -\varphi(\x_{i+1})\right) dt\\
      &\qquad + \int_0^T \sic\Bigl[\varphi(\x_{i+1})\left(c\dot{\x}_{i+1}-\rho_i^h\dot{\x}_i\right)+(\rho_i^h-c)\,\varphi(\x_{i})\dot{\x}_{i}\Bigr]\, dt\,,\\
    I_i^2 &=-\int_0^T\sic\left(\vartheta(\rho_i^h)-\vartheta(c)\right)\Bigl[\sfF_{\hat\rho^h}(\x_{i+1})\,\varphi(\x_{i+1})-\sfF_{\hat\rho^h}(\x_i)\,\varphi(\x_i)\Bigr]\,dt\\
     &\qquad+\int_0^T\sic\,\vartheta(\rho_i^h)\,\int_{K_i}\partial_x \sfF_{\hat\rho^h}\,\varphi \,dx \,dt\,,
\end{align*}
and
\[
    B^1 = \int_0^T \varphi(\x_N)\bigl[c\dot\x_N +\vartheta(c)\,\sfF_{\hat\rho^h}(\x_N)\,\varphi (\x_N)\bigr] \,dt  -\int_0^T \varphi(\x_0)\bigl[c\dot\x_0 + \vartheta(c)\,\sfF_{\hat\rho^h}(\x_0)\,\varphi(\x_0)\bigr]\,dt\,.
\]
This allows to rewrite the right-hand side of \eqref{eq:entropy_disc} as
\begin{align*}
    &\liminf_{h\to 0}\sum_{i=0}^{N-1} I_i^1 + I_i^2 = \liminf_{h\to 0}\sum_{i=0}^{N-1} A_i^1 +I_i^3 + I_i^4
    \end{align*}
where $A_i^1$ is defined in \eqref{eq:A1} and
    \begin{align*}
    I_i^3  &=\int_0^T \sic\Bigl[\varphi(\x_{i+1})\left(c\dot{\x}_{i+1}-\rho_i^h\dot{\x}_i\right)+(\rho_i^h-c)\,\varphi(\x_{i})\dot{\x}_{i}\Bigr]\, dt\\
    &\qquad-\int_0^T\sic\left(\vartheta(\rho_i^h)-\vartheta(c)\right)\Bigl[\sfF_{\hat\rho^h}(\x_{i+1})\,\varphi(\x_{i+1})-\sfF_{\hat\rho^h}(\x_i)\,\varphi(\x_i)\Bigr]\,dt\,,\\
    I_i^4 &= \int_0^T\sic\,\vartheta(\rho_i^h)\,\int_{K_i}(\partial_x \sfF_{\hat\rho^h})\,\varphi \,dx \,dt\,.
\end{align*}
In Lemma \ref{A1} we have seen that 
\begin{align*}
    \left|\sum_{i=0}^{N-1}A_i^1\right|&\to 0\quad \mbox{as}\quad h\to 0.
\end{align*}
A numerical integration by parts yields
\begin{align*}
    \sum_{i=0}^{N-1} I_i^3 &= \sum_{i=0}^{N-1}\int_0^T\sic\,\varphi(\x_i)\left[\rho_i^h\dot{\x}_i+\vartheta(\rho_i^h)\,\sfF_{\hat\rho^h}(\x_i)\right] dt\\
    & -\sum_{i=1}^{N}\int_0^T\sicm\,\varphi(\x_i)\left[\rho_{i-1}^h\dot{\x}_{i-1}+\vartheta(\rho_{i-1}^h)\,\sfF_{\hat\rho^h}(\x_i)\right] dt\\
    & +\sum_{i=0}^{N}\int_0^T\left(\sicm-\sic\right)\varphi(\x_i)\left[c\dot{\x}_i+\vartheta(c)\,\sfF_{\hat\rho^h}(\x_i)\right] dt - B^1\\
    & = -B^1 + \sum_{i=1}^{N-1}I_i^5+A_i^2+A_i^3\,,
\end{align*}
where the contribution of the terms $A_i^2$ introduced in \eqref{eq:A2} vanishes thanks to Lemma~\ref{A2}, while, using  Lemma~\ref{A3}, we have that
\begin{align*}
    \liminf_{h\to 0} \sum_{i=1}^{N-1} A_i^3  & \geq 0\,.
\end{align*}
The contributions of $I_i^4$ and
\[
I_i^5=- \int_0^T\sic\,\varphi(\x_i)\,\vartheta(\rho_{i}^h)\left(\sfF_{\hat\rho^h}(\x_{i+1})-\sfF_{\hat\rho^h}(\x_i)\right) dt\,,
\]
merge into $A_i^4$ introduced in \eqref{eq:A4}, which was shown to vanish due to Lemma~\ref{A4}. Putting all the terms together a passing to the limit inferior then yields \eqref{eq:entropy_disc}. 

We now prove that \eqref{eq:entropy_disc} ultimately leads to the entropy inequality \eqref{eq:EI} in Definition~\ref{def:entropy_sol}. Indeed, due to the almost sure convergence and $L^1([0,T]\times\R)$ convergence of $\hat\rho^h$ towards $\rho$ as $h\to 0$, along with the convergences
\[
    \sfF_{\hat\rho^h} \to \sfF_{\rho},\qquad \partial_x\sfF_{\hat\rho^h}\to \partial_x\sfF_{\rho}\qquad\text{for almost every $(t,x)\in(0,T)\times\R$}\,,
\]
one observes that the right-hand side of inequality \eqref{eq:entropy_disc} coincides with the right-hand side of \eqref{eq:EI}, which then concludes the proof.\hfill \qedsymbol

\appendix

\section{Proof of Lemma~\ref{lem:liploc_DPA}}\label{app:liploc_DPA}

For notational simplicity, we set
\[
    \f_i^V = V'(x_i),\qquad \f_i^W = h\sum_{j\ne i} W'(\x_i-\x_j)\,.
\]
Due to the regularity assumed on $V$ and $W$, and the lower bound \eqref{eq:lower_DPA}, we easily find
\begin{align*}
    |\f_{i+1}^V - \f_i^V| &\le \|V''\|_{L^\infty(\R)}|K_i|\,,\\
    |\f_{i+1}^W - \f_i^W| &\le \Bigl(m\|W''\|_{L^\infty(\R\setminus\{0\})} + 2M\|W'\|_{L^\infty(\R\setminus\{0\})}\Bigr)|K_i|\,.
\end{align*}
An application of the triangle inequality then yields
\[
    |\f_{i+1}-\f_i| \le c_{\f}^{(1)}|K_i|\,,
\]
with the constant
\[
    c_{\f}^{(1)} := \|V''\|_{L^\infty(\R)} + m\|W''\|_{L^\infty(\R\setminus\{0\})} + 2M\|W'\|_{L^\infty(\R\setminus\{0\})}\,.
\]

As for the other inequality, we make use of the general inequality
\[
    |g(r+a) - 2g(r) + g(r-b)| \le \Lip(g')\left(\frac{a^2 + b^2}{2}\right) + |g'(r)||a-b|\,,\qquad a,b>0\,,
\]
whenever the terms are well-defined. For the external potential, we have
\begin{align*}
    |\f_{i+1}^V-2\f_i^V+\f_{i-1}^V| &= |V'(\x_i + |K_i|) - 2V'(\x_i) + V'(\x_i - |K_i|)| \\
    &\le \Lip(V'')\left(\frac{|K_i|^2 + |K_{i-1}|^2}{2}\right) + \|V''\|_{L^\infty(\R)}\Bigl||K_i|-|K_{i-1}|\Bigr|\,.
\end{align*}
As for the interaction potential, we find
\begin{align*}
    |\f_{i+1}^W-2\f_i^W+\f_{i-1}^W| &\le h\sum_{j\ne i+1,i,i-1} \Bigl| W'(\x_i-\x_j +|K_i|) - 2hW'(\x_i-\x_j) + W'(\x_i-\x_j-|K_{i-1}|)\Bigr| \\
    &\qquad + 2h\bigl|W'(|K_i|) - W'(|K_{i-1}|)\bigr| \\
    &= h\sum_{j> i+1} \Bigl| W'(\x_j-\x_i -|K_i|) - 2hW'(\x_j-\x_i) + W'(\x_j-\x_i+|K_{i-1}|)\Bigr| \\
    &\qquad +h\sum_{j<i-1} \Bigl| W'(\x_i-\x_j +|K_i|) - 2hW'(\x_i-\x_j) + W'(\x_i-\x_j-|K_{i-1}|)\Bigr| \\
    &\qquad + 2h\bigl|W'(|K_i|) - W'(|K_{i-1}|)\bigr| \\
    &\le m\Lip(W'')\left(\frac{|K_i|^2 + |K_{i-1}|^2}{2}\right) + m\|W''\|_{L^\infty(\R\setminus\{0\})}\Bigl||K_i|-|K_{i-1}|\Bigr| \\
    &\qquad + 2h\|W''\|_{L^\infty(\R\setminus\{0\})}\Bigl||K_i|-|K_{i-1}|\Bigr|\,.
\end{align*}
Putting the estimates together, we arrive at
\[
    |\f_{i+1}-2\f_i+\f_{i-1}| \le c_{\f}^{(2)} \Bigl(|K_i|^2 + |K_{i-1}|^2\Bigr) + c_{\f}^{(3)}\Bigl||K_i|-|K_{i-1}|\Bigr|\,,
\]
with
\[
    c_{\f}^{(2)} := \Lip(V'') + m\Lip(W'')\,,\qquad c_{\f}^{(3)} := \|V''\|_{L^\infty(\R)} + (2+m)\|W''\|_{L^\infty(\R\setminus\{0\})}\,.
\]
We then conclude the first part of the lemma by setting
\[
    c_{\f} := \max\{c_{\f}^{(1)},c_{\f}^{(2)},c_{\f}^{(3)}\}
\]

For the Newtonian case, i.e.\ $W$ satisfying (A-Wn), we simply obtain $\f_{i+1}^W-\f_i^W = 2h$, and hence
\[
    |\f_{i+1}^W-\f_i^W| = 2h \le 2M|K_i|\,,\qquad |\f_{i+1}^W-2\f_i^W+\f_{i-1}^W| = 0\,,
\]
which ultimately leads to the required inequality with $c_{\f}:=\max\{\|V''\|_{L^\infty(\R)}+2M,\Lip(V'')\}$.

\bibliographystyle{myplain}

\begin{thebibliography}{10}

\bibitem{AGS2008}
L.~Ambrosio, N.~Gigli, and G.~Savar\'{e}.
\newblock {\em Gradient flows in metric spaces and in the space of probability
  measures}.
\newblock Lectures in Mathematics ETH Z\"{u}rich. Birkh\"{a}user Verlag, Basel,
  second edition, 2008.

\bibitem{AMS2011}
L.~Ambrosio, E.~Mainini, and S.~Serfaty.
\newblock Gradient flow of the {C}hapman-{R}ubinstein-{S}chatzman model for
  signed vortices.
\newblock {\em Ann. Inst. H. Poincar\'{e} Anal. Non Lin\'{e}aire},
  28(2):217--246, 2011.

\bibitem{AS2008}
L.~Ambrosio and S.~Serfaty.
\newblock A gradient flow approach to an evolution problem arising in
  superconductivity.
\newblock {\em Communications on Pure and Applied Mathematics},
  61(11):1495--1539, 2008.

\bibitem{BHJ11}
M.~Baines, M.~Hubbard, and P.~Jimack.
\newblock A moving mesh finite element algorithm for the adaptive solution of
  time-dependent partial differential equations with moving boundaries.
\newblock {\em Journal of Computational and Applied Mathematics}, 54:450--469,
  2004.

\bibitem{BLL2012}
A.~L. Bertozzi, T.~Laurent, and F.~L\'{e}ger.
\newblock Aggregation and spreading via the {N}ewtonian potential: the dynamics
  of patch solutions.
\newblock {\em Math. Models Methods Appl. Sci.}, 22(suppl. 1):1140005, 39,
  2012.

\bibitem{BBL2005}
F.~Bolley, Y.~Brenier, and G.~Loeper.
\newblock Contractive metrics for scalar conservation laws.
\newblock {\em J. Hyperbolic Differ. Equ.}, 2(1):91--107, 2005.

\bibitem{BCdiFP2015}
G.~A. Bonaschi, J.~A. Carrillo, M.~Di~Francesco, and M.~A. Peletier.
\newblock Equivalence of gradient flows and entropy solutions for singular
  nonlocal interaction equations in 1{D}.
\newblock {\em ESAIM Control Optim. Calc. Var.}, 21(2):414--441, 2015.

\bibitem{Bre2009}
Y.~Brenier.
\newblock {$L^2$} formulation of multidimensional scalar conservation laws.
\newblock {\em Arch. Ration. Mech. Anal.}, 193(1):1--19, 2009.

\bibitem{BHR96}
C.~Budd, W.~Huang, and R.~Russell.
\newblock Moving mesh methods for problems with blow-up.
\newblock {\em SIAM Journal on Scientific Computing}, 17:305--327, 1996.

\bibitem{BHR09}
C.~Budd, W.~Huang, and R.~Russell.
\newblock Adaptivity with moving grids.
\newblock {\em Acta Numerica}, 18:111--241, 2009.

\bibitem{BdiFD2006}
M.~Burger, M.~Di~Francesco, and Y.~Dolak-Struss.
\newblock The {K}eller-{S}egel model for chemotaxis with prevention of
  overcrowding: linear vs. nonlinear diffusion.
\newblock {\em SIAM J. Math. Anal.}, 38(4):1288--1315, 2006.

\bibitem{CHR02}
W.~Cao, W.~Huang, and R.~Russell.
\newblock A moving mesh method based on the geometric conservation law.
\newblock {\em SIAM Journal on Scientific Computing}, 24:118--142, 2002.

\bibitem{CHR03}
W.~Cao, W.~Huang, and R.~Russell.
\newblock Approaches for generating moving adaptive meshes: location versus
  velocity.
\newblock {\em Applied Numerical Mathematics}, 47(2):121--138, 2003.

\bibitem{CdiFFLS2011}
J.~A. Carrillo, M.~DiFrancesco, A.~Figalli, T.~Laurent, and D.~Slep\v{c}ev.
\newblock Global-in-time weak measure solutions and finite-time aggregation for
  nonlocal interaction equations.
\newblock {\em Duke Math. J.}, 156(2):229--271, 2011.

\bibitem{CDMM}
J.~A. Carrillo, B.~D{\"u}ring, D.~Matthes, and D.~S. McCormick.
\newblock A lagrangian scheme for the solution of nonlinear diffusion equations
  using moving simplex meshes.
\newblock {\em Journal of Scientific Computing}, 75(3):1463--1499, 2018.

\bibitem{CG-CV2019}
J.~A. Carrillo, D.~Gómez-Castro, and J.~L. Vázquez.
\newblock A fast regularisation of a newtonian vortex equation.
\newblock {\em Preprint on arXiv (https://arxiv.org/abs/1912.00912)}, 2019.

\bibitem{CG-CV2020}
J.~A. Carrillo, D.~Gómez-Castro, and J.~L. Vázquez.
\newblock Vortex formation for a non-local interaction model with newtonian
  repulsion and superlinear mobility.
\newblock {\em Preprint on arXiv (https://arxiv.org/abs/2007.01185)}, 2020.

\bibitem{CHPW}
J.~A. Carrillo, Y.~Huang, F.~S. Patacchini, and G.~Wolansky.
\newblock Numerical study of a particle method for gradient flows.
\newblock {\em Kin. Rel. Mod.}, 10(3):613--641, 2017.

\bibitem{CLSS2010}
J.~A. Carrillo, S.~Lisini, G.~Savar\'{e}, and D.~Slep\v{c}ev.
\newblock Nonlinear mobility continuity equations and generalized displacement
  convexity.
\newblock {\em J. Funct. Anal.}, 258(4):1273--1309, 2010.

\bibitem{CRP2013}
J.~A. Carrillo, S.~Martin, and V.~Panferov.
\newblock A new interaction potential for swarming models.
\newblock {\em Phys. D}, 260:112--126, 2013.

\bibitem{CdiFL2006}
J.~A. Carrillo, M.~Di~Francesco, and C.~Lattanzio.
\newblock Contractivity of {W}asserstein metrics and asymptotic profiles for
  scalar conservation laws.
\newblock {\em J. Differential Equations}, 231(2):425--458, 2006.

\bibitem{CMW2021}
J.~A. Carrillo, D.~Matthes, and M.-T. Wolfram.
\newblock Chapter 4 - lagrangian schemes for wasserstein gradient flows.
\newblock In A.~Bonito and R.~H. Nochetto, editors, {\em Geometric Partial
  Differential Equations - Part II}, volume~22 of {\em Handbook of Numerical
  Analysis}, pages 271--311. Elsevier, 2021.

\bibitem{CR2006}
F.~A. C.~C. Chalub and J.~F. Rodrigues.
\newblock A class of kinetic models for chemotaxis with threshold to prevent
  overcrowding.
\newblock {\em Port. Math. (N.S.)}, 63(2):227--250, 2006.

\bibitem{diFFRo2017}
M.~Di~Francesco, S.~Fagioli, and M.~D. Rosini.
\newblock Deterministic particle approximation of scalar conservation laws.
\newblock {\em Boll. Unione Mat. Ital.}, 10(3):487--501, 2017.

\bibitem{diFSt}
M.~Di~Francesco and G.~Stivaletta.
\newblock Convergence of the follow-the-leader scheme for scalar conservation
  laws with space dependent flux.
\newblock {\em Dis. Cont. Dyn. Syst.}, 40(1):233--266, 2020.

\bibitem{diFra2016}
M.~Di~Francesco.
\newblock Scalar conservation laws seen as gradient flows: known results and
  new perspectives.
\newblock In {\em Gradient flows: from theory to application}, volume~54 of
  {\em ESAIM Proc. Surveys}, pages 18--44. EDP Sci., Les Ulis, 2016.

\bibitem{diFFRa2019}
M.~Di~Francesco, S.~Fagioli, and E.~Radici.
\newblock Deterministic particle approximation for nonlocal transport equations
  with nonlinear mobility.
\newblock {\em J. Differential Equations}, 266(5):2830--2868, 2019.

\bibitem{diFR2008}
M.~Di~Francesco and J.~Rosado.
\newblock Fully parabolic {K}eller-{S}egel model for chemotaxis with prevention
  of overcrowding.
\newblock {\em Nonlinearity}, 21(11):2715--2730, 2008.

\bibitem{DNS2009}
J.~Dolbeault, B.~Nazaret, and G.~Savar\'{e}.
\newblock A new class of transport distances between measures.
\newblock {\em Calc. Var. Partial Differential Equations}, 34(2):193--231,
  2009.

\bibitem{DGM2013}
W.~Dreyer, C.~Guhlke, and R.~Müller.
\newblock Overcoming the shortcomings of the nernst–planck model.
\newblock {\em Phys. Chem. Chem. Phys.}, 15:7075--7086, 2013.

\bibitem{DB2015}
L.~Dyson and R.~E. Baker.
\newblock The importance of volume exclusion in modelling cellular migration.
\newblock {\em Journal of Mathematical Biology}, 71(3):691--711, 2015.

\bibitem{EPSS2021}
A.~Esposito, F.~S. Patacchini, A.~Schlichting, and D.~Slep{\v c}ev.
\newblock Nonlocal-interaction equation on graphs: Gradient flow structure and
  continuum limit.
\newblock {\em Archive for Rational Mechanics and Analysis}, 2021.

\bibitem{EGO2016}
E.~Esselborn, N.~Gigli, and F.~Otto.
\newblock Algebraic contraction rate for distance between entropy solutions of
  scalar conservation laws.
\newblock {\em J. Math. Anal. Appl.}, 435(2):1525--1551, 2016.

\bibitem{FRa2018}
S.~Fagioli and E.~Radici.
\newblock Solutions to aggregation-diffusion equations with nonlinear mobility
  constructed via a deterministic particle approximation.
\newblock {\em Math. Models Methods Appl. Sci.}, 28(9):1801--1829, 2018.

\bibitem{GL1998}
G.~Giacomin and J.~L. Lebowitz.
\newblock Phase segregation dynamics in particle systems with long range
  interactions. {II}. {I}nterface motion.
\newblock {\em SIAM J. Appl. Math.}, 58(6):1707--1729, 1998.

\bibitem{GT1}
L.~Gosse and G.~Toscani.
\newblock Identification of asymptotic decay to self-similarity for one-
  dimensional filtration equations.
\newblock {\em SIAM J. Numer. Anal.}, 43:2590–2606, 2006.

\bibitem{GT2}
L.~Gosse and G.~Toscani.
\newblock Lagrangian numerical approximations to one-dimensional
  convolution-diffusion equations.
\newblock {\em SIAM J. Sci. Comput.}, 28:1203–1227, 2006.

\bibitem{Kani1995}
G.~Kaniadakis.
\newblock Generalized {B}oltzmann equation describing the dynamics of bosons
  and fermions.
\newblock {\em Phys. Lett. A}, 203(4):229--234, 1995.

\bibitem{KR2003}
K.~H. Karlsen and N.~H. Risebro.
\newblock On the uniqueness and stability of entropy solutions of nonlinear
  degenerate parabolic equations with rough coefficients.
\newblock {\em Discrete Contin. Dyn. Syst.}, 9(5):1081--1104, 2003.

\bibitem{Kru1970}
S.~N. Kru\v{z}kov.
\newblock First order quasilinear equations with several independent variables.
\newblock {\em Mat. Sb. (N.S.)}, 81 (123):228--255, 1970.

\bibitem{LZ2000}
F.~Lin and P.~Zhang.
\newblock On the hydrodynamic limit of {G}inzburg-{L}andau vortices.
\newblock {\em Discrete Contin. Dynam. Systems}, 6(1):121--142, 2000.

\bibitem{LM2010}
S.~Lisini and A.~Marigonda.
\newblock On a class of modified {W}asserstein distances induced by concave
  mobility functions defined on bounded intervals.
\newblock {\em Manuscripta Math.}, 133(1-2):197--224, 2010.

\bibitem{Mielke2016}
A.~Mielke.
\newblock On evolutionary {$\varGamma$}-convergence for gradient systems.
\newblock In {\em Macroscopic and large scale phenomena: coarse graining, mean
  field limits and ergodicity}, volume~3 of {\em Lect. Notes Appl. Math.
  Mech.}, pages 187--249. Springer, [Cham], 2016.

\bibitem{Ott1999}
F.~Otto.
\newblock Evolution of microstructure in unstable porous media flow: a
  relaxational approach.
\newblock {\em Comm. Pure Appl. Math.}, 52(7):873--915, 1999.

\bibitem{PRST2020}
M.~A. Peletier, R.~Rossi, G.~Savaré, and O.~Tse.
\newblock Jump processes as generalized gradient flows.
\newblock {\em Preprint on arXiv (https://arxiv.org/abs/2006.10624)}, 2020.

\bibitem{RoSa2003}
R.~Rossi and G.~Savar{\'e}.
\newblock Tightness, integral equicontinuity and compactness for evolution
  problems in {B}anach spaces.
\newblock {\em Ann. Sc. Norm. Super. Pisa Cl. Sci. (5)}, 2(2):395--431, 2003.

\bibitem{SS2004}
E.~Sandier and S.~Serfaty.
\newblock Gamma-convergence of gradient flows with applications to
  {G}inzburg-{L}andau.
\newblock {\em Comm. Pure Appl. Math.}, 57(12):1627--1672, 2004.

\bibitem{SV2014}
S.~Serfaty and J.~L. V\'{a}zquez.
\newblock A mean field equation as limit of nonlinear diffusions with
  fractional {L}aplacian operators.
\newblock {\em Calc. Var. Partial Differential Equations}, 49(3-4):1091--1120,
  2014.

\bibitem{Slep2008}
D.~Slep\v{c}ev.
\newblock Coarsening in nonlocal interfacial systems.
\newblock {\em SIAM J. Math. Anal.}, 40(3):1029--1048, 2008.

\bibitem{SMR01}
J.~Stockie, J.~Mackenzie, and R.~Russell.
\newblock A moving mesh method for one-dimensional hyperbolic conservation
  laws.
\newblock {\em SIAM Journal on Scientific Computing}, 22:1791--1813, 2001.

\bibitem{Vol1967}
A.~I. Vol'pert.
\newblock Spaces {${\rm BV}$} and quasilinear equations.
\newblock {\em Mat. Sb. (N.S.)}, 73 (115):255--302, 1967.

\end{thebibliography}

\end{document}